\documentclass[11pt]{article}
\usepackage[T1]{fontenc}
\usepackage[utf8]{inputenc}
\usepackage{lmodern}
\usepackage[a4paper,margin=28mm]{geometry}
\usepackage{amsmath,amssymb,amsthm,mathtools}
\usepackage{microtype,booktabs,array,tabularx}
\usepackage{tikz}
\usepackage[hidelinks]{hyperref}
\hypersetup{pdftitle={Multivariate stability of powered triangular recurrences and parametrised Eulerian polynomials}}
\usepackage[nameinlink,capitalize,noabbrev]{cleveref}
\usepackage{enumitem}
\newtheorem{theorem}{Theorem}[section]
\newtheorem*{unnumtheorem}{Theorem}
\newtheorem{maintheorem}{Theorem}

\crefname{maintheorem}{Theorem}{Theorems}
\Crefname{maintheorem}{Theorem}{Theorems}
\newtheorem{mainconjecture}{Conjecture}
\crefname{mainconjecture}{Conjecture}{Conjectures}
\Crefname{mainconjecture}{Conjecture}{Conjectures}
\newtheorem{proposition}[theorem]{Proposition}
\newtheorem{lemma}[theorem]{Lemma}

\theoremstyle{definition}
\newtheorem{definition}[theorem]{Definition}
\newtheorem{example}[theorem]{Example}
\theoremstyle{remark}
\newtheorem{remark}[theorem]{Remark}

\newcommand{\SSS}{\mathfrak{S}}

\DeclareMathOperator{\des}{des}
\DeclareMathOperator{\DES}{DES}
\DeclareMathOperator{\exc}{exc}
\DeclareMathOperator{\EXC}{EXC}
\DeclareMathOperator{\bl}{bl}
\DeclareMathOperator{\Comp}{Comp}

\DeclareMathOperator{\wt}{wt}
\numberwithin{equation}{section}
\tikzset{treevertex/.style={circle,draw,inner sep=1.5pt,minimum size=5.5mm,font=\small},treeedge/.style={line width=.5pt}}
\title{Multivariate stability of powered triangular recurrences\\and parametrised Eulerian polynomials}
\author{Umesh Shankar\\
    Department of Computer Science and Automation,\\
    Indian Institute of Science Bengaluru,\\
    Bengaluru 560012, Karnataka, India\\
    Email: \texttt{umeshshankar@outlook.com}
}
\date{\today}
\begin{document}
\maketitle

\begin{abstract}
We prove multivariate stability for a class of affine triangular
recurrences whose two coefficients are raised to an arbitrary positive
integer power.  The variables keep track of cross-step heights in a weighted
lattice-path model. The applications include stability and real-rootedness of parametrised Eulerian,
Stirling, and Lah families. For the parametrised Eulerian polynomials,
we obtain multivariate refinements of known real-rootedness results and
prove simplicity and strict interlacing of consecutive rows.
A kernel-preserving bijection with increasing binary trees yields a
synchronised Foata--Strehl action on pairs and a combinatorial interpretation of
their gamma coefficients along with a recurrence.
The same bijection identifies refinement posets that admit symmetric
Boolean and chain decompositions for parameter one and cover-supported
$\mathfrak{sl}_2$ operators prove their PECKness.  We extend the
stability and strict interlacing results to synchronised Stirling
permutations and prove PECKness for ordinary Stirling permutations
with a fixed number of plateaux.
\end{abstract}

\medskip
\noindent\textbf{Keywords:} stable polynomial; triangular recurrence;
Eulerian polynomial; Stirling permutation; gamma-positivity;
symmetric chain decomposition.

\section*{Introduction}
\renewcommand{\theequation}{I.\arabic{equation}}
For a positive integer $n$, let $[n]=\{1,2,\dots,n\}$ and let
$\SSS_n$ denote the set of permutations of $[n]$.
For $\pi=\pi_1\pi_2\cdots\pi_n\in\SSS_n$, let
$\DES(\pi)=\{i\in[n-1]:\pi_i>\pi_{i+1}\}$ be its descent set and let
$\des(\pi)=|\DES(\pi)|$.
Similarly, let $\EXC(\pi)=\{i\in[n]:\pi_i>i\}$ and
$\exc(\pi)=|\EXC(\pi)|$ be its excedance set and number of excedances.
The $n$th Eulerian polynomial is
\[
 A_n(t)=\sum_{\pi\in\SSS_n}t^{\des(\pi)}.
\]
It also enumerates permutations by excedances:
$A_n(t)=\sum_{\pi\in\SSS_n}t^{\exc(\pi)}$.
These polynomials and their many interpretations are treated in
Petersen's book~\cite{Petersen2015}.

A polynomial $p(t)=\sum_{k=0}^d a_kt^k$ is \emph{palindromic} if
$a_k=a_{d-k}$ for every $1\le k\le d$.  Every such polynomial has a unique expansion
\[
 p(t)=\sum_{j=0}^{\lfloor d/2\rfloor}
       \gamma_jt^j(1+t)^{d-2j}.
\]
It is \emph{gamma-positive} if all $\gamma_j$ are nonnegative.
A real polynomial is \emph{real-rooted} if all its zeros are real.
The ordinary Eulerian polynomials are palindromic, gamma-positive,
and real-rooted~\cite{Petersen2015}.
Multivariate refinements retain additional information about the
descent statistic.  A polynomial with real coefficients is
\emph{real stable} if it is nonzero whenever all its variables lie
in the open upper half-plane.
The following result is one such refinement due to Haglund and Visontai \cite{haglund-visontai-stable-multivariate}.

\begin{unnumtheorem}[{\cite[Theorem~3.2]{haglund-visontai-stable-multivariate}}]
Put $\pi_0=\pi_{n+1}=0$ for $\pi\in\SSS_n$.  The polynomial
\[
 \mathcal E_n(\mathbf x,\mathbf y)
 =\sum_{\pi\in\SSS_n}
   \prod_{\substack{1\le i\le n\\\pi_i>\pi_{i+1}}}x_{\pi_i}
   \prod_{\substack{1\le i\le n\\\pi_{i-1}<\pi_i}}y_{\pi_i}
\]
is real stable.
\end{unnumtheorem}
The terminal descent is included in this convention, so setting
$x_i=t$ and $y_i=1$ gives $tA_n(t)$.

The Eulerian numbers $A(n,k)$, which are the coefficients of $A_n(t)$,
satisfy
\[
 A(n,k)=(k+1)A(n-1,k)+(n-k)A(n-1,k-1)
 \qquad(n\ge2),
\]
with $A(1,0)=1$ and zero entries outside $0\le k\le n-1$.
In~\cite{Shankar2026}, an $\ell$-parametrised generalisation was
defined by raising the two recurrence coefficients to the positive
integer power $\ell$.  These numbers count $\ell$-tuples of
subexceedant functions with a common block-leader set.
Shankar proved palindromicity and log-concavity, established
gamma-positivity for $\ell=2$, and proposed real-rootedness and
gamma-positivity for the remaining powers.

We write $A_{n,\ell}(t)=\sum_k A_\ell(n,k)t^k$, where
\[
 A_\ell(n,k)=(k+1)^\ell A_\ell(n-1,k)
             +(n-k)^\ell A_\ell(n-1,k-1),\qquad A_\ell(1,0)=1,
\]
and entries outside $0\le k\le n-1$ are zero.  Palindromicity gives
the unique expansion
\begin{equation}\label{alg:gamma-expansion}
 A_{n,\ell}(t)=\sum_{j=0}^{\lfloor(n-1)/2\rfloor}
 \gamma_{n,\ell,j}t^j(1+t)^{n-1-2j}.
\end{equation}
Write $\Gamma_{n,\ell}(u)=\sum_j\gamma_{n,\ell,j}u^j$ for its
gamma-polynomial.  The following univariate results are due to
Alexandersson \cite{Alexandersson2026}.

\begin{unnumtheorem}[{Alexandersson~\cite[Theorem~5.1]{Alexandersson2026}}]
For positive integers $n$ and $\ell$, both $A_{n,\ell}$ and
$\Gamma_{n,\ell}$ are real-rooted, and $\Gamma_{n,\ell}$ has nonnegative
coefficients.  The rows $A_{n,\ell}$ and $A_{n+1,\ell}$ weakly interlace.
\end{unnumtheorem}
Alexandersson's proof divides the degree-$k$ coefficient by $\binom{n-1}{k}^{\ell}$ and restores the original coefficients by a Hadamard-product argument. A similar normalisation argument also appears in \cite{Shankar2025}.  Our main stability theorem extends this normalisation to powered affine recurrences and a polynomial that keeps track of a separate variable for each cross-step height.
For the Eulerian specialisation, we also prove simplicity and strict
interlacing.  The tree actions and refinement posets provide
combinatorial structures for these polynomial families.

Stability criteria for recursive polynomials and their interlacing
consequences were developed by Ding and Zhu~\cite{DingZhu2024}.
Yang and Zhang~\cite{YangZhang2025} constructed stable refinements of
Narayana and second-order Eulerian polynomials using labelled plane
trees.  Our variables record cross-step heights in the insertion
history.  The six sections below develop the relevant combinatorial models,
the general stability theorem, and its Eulerian and Stirling applications.
We state the main results here.

A subexceedant function $f:[n]\to[n]$ is a function that satisfies $f(i)\le i$. Its kernel is the partition into classes sharing the same image, and $\bl(f)$ is the set of their minima.  Let $\mathcal Q_{n,\ell}$ consist of
$\ell$-tuples with a common set $B$ of block minima, ordered by proper
kernel refinement in every coordinate, with equality added.
Let $\mathcal T_{n,\ell}$ consist of $\ell$-tuples of increasing plane
binary trees on $[n]$ with a common set $L$ of left-child labels.

In Section \ref{sec:preliminaries}, we construct the bijection that takes the tuples of subexceedant functions counted by parametrised Eulerian numbers to tuples of synchronised increasing plane binary trees. Additionally, we also give a natural poset on these objects.
\begin{maintheorem}[Combinatorial models]\label{main:models}
\label{tree:kernel-bijection}\label{poset:grading}
\begin{enumerate}[label=\textup{(\roman*)}]
\item There is a bijection $\Phi_n:\mathcal Q_{n,1}\to\mathcal T_{n,1}$
under which the blocks of $\ker(f)$ are the vertex sets of the maximal
right chains of $\Phi_n(f)$.  In particular,
$L(\Phi_n(f))=\bl(f)\setminus\{1\}$.  Its coordinatewise extension
is a bijection $\mathcal Q_{n,\ell}\to\mathcal T_{n,\ell}$, and
\begin{equation}\label{tree:enumerator}
 A_{n,\ell}(t)=\sum_{\mathbf T\in\mathcal T_{n,\ell}}t^{|L(\mathbf T)|}.
\end{equation}
\item The poset $\mathcal Q_{n,\ell}$ is bounded and graded of rank
$n-1$, with rank function $\rho(\mathbf f)=|B(\mathbf f)|-1$
and rank-generating polynomial $A_{n,\ell}(t)$.
\end{enumerate}
\end{maintheorem}

Our next main result is a stability theorem. The main stability theorem concerns the affine triangular recurrences from
\cite{Shankar2026}.  Put
\[
 c(n,k)=\alpha n+\beta k+\gamma,\qquad
 d(n,k)=\alpha'n+\beta'k+\gamma',
\]
and let \[T_\ell(n,k)=c(n,k)^\ell T_\ell(n-1,k)
+d(n,k)^\ell T_\ell(n-1,k-1)\] start with $T_\ell(n_0,k_0)=1$.
Use the paths of Shankar~\cite[Section~2]{Shankar2026} from
$(k_0,n_0)$ with north steps $(0,1)$ and cross steps $(1,1)$. Let $\mathcal{P}_n$ be the set of paths with north and cross steps that starts at $(k_0,n_0)$ and ends at height $n$. A step ending at $(k,i)$ has weight
$c(i,k)$ or $d(i,k)$, respectively.  Write $\wt(p)$ for their product
and $C(p)$ for the set of heights where the cross steps occur. Different labeling conventions for height do not affect the stability result. With
\[
 a_n(j)=c(n,k_0+j),\quad b_n(j)=d(n,k_0+j+1),\quad
 w_n(j)=\prod_{h=0}^{j-1}\frac{b_n(h)}{a_n(h)},
\]
define the two cross-step refinements
\[
 \mathcal F_{n,\ell}=\sum_p \wt(p)^\ell\prod_{i\in C(p)}z_i,
 \qquad
 \mathcal P_{n,\ell}=\sum_p
 \frac{\wt(p)^\ell}{w_n(|C(p)|)^\ell}\prod_{i\in C(p)}z_i,
\]
where the sums run over paths to height $n$ and both initial polynomials
are $1$.

\begin{maintheorem}[Stability of powered affine recurrences]\label{stab:general}
Assume
\[
 \alpha,\alpha',\beta\ge0,\qquad -\alpha'\le\beta'\le0,
\]
and $a_n(j),b_n(j)>0$ for $n>n_0$ and $0\le j\le n-n_0-1$.
Suppose also that
\[
 \alpha/\beta\in\mathbb Z_{\ge0}\quad\hbox{if }\beta>0,
 \qquad
 \alpha'/(-\beta')\in\mathbb Z_{\ge1}\quad\hbox{if }\beta'<0.
\]
For every integer $\ell\ge1$ and $n\ge n_0$, both
$\mathcal P_{n,\ell}$ and $\mathcal F_{n,\ell}$ are real stable.
In particular, $\sum_{j=0}^{n-n_0}T_\ell(n,k_0+j)t^j$ has only
negative real zeros.
\end{maintheorem}
The proof in Section~\ref{stab:section} balances the two path extensions
and applies finite P\'olya--Schur theory and hypergeometric zero
locations.  Table~\ref{stab:applications} lists the covered sequence
families. We also remark on some extensions to general triangular arrays.

For the Eulerian recurrence, write $F_{n,\ell}$ for the cross-step
refinement defined above with initial vertex $(0,1)$. Thus, a variable $z_i$ marks a common block minimum $i>1$.  Write $P_{n,\ell}$ for the polynomial
obtained by dividing its degree-$k$ coefficients by $\binom{n-1}{k}^\ell$,
and put $\widehat A_{n,\ell}(t)=P_{n,\ell}(t,\ldots,t)$.

\begin{maintheorem}[Eulerian stability and strict interlacing]\label{main:eulerian}
\label{stab:main}\label{alg:real-rootedness}\label{int:main}
\label{int:row-pencil}
For all positive integers $n,\ell$, the following assertions hold.
\begin{enumerate}[label=\textup{(\roman*)}]
\item Both $F_{n,\ell}$ and $P_{n,\ell}$ are real stable.
The polynomials $A_{n,\ell}$, $\widehat A_{n,\ell}$, and
$F_{n,\ell}(\lambda_2t,\ldots,\lambda_nt)$ for $\lambda_i>0$
have only negative real zeros.  The common block-minimum set of a
uniformly chosen synchronised tuple, with $1$ removed, has a strongly
Rayleigh distribution.
\item For fixed $\ell$, $A_{n,\ell}$ has simple zeros for $n\ge2$.
For $n\ge2$, if $\alpha_1<\cdots<\alpha_{n-1}$ and $\beta_1<\cdots<\beta_n$
are the zeros of $A_{n,\ell}$ and $A_{n+1,\ell}$, then
\begin{equation}\label{int:strict-chain}
 \beta_1<\alpha_1<\beta_2<\cdots<\beta_{n-1}<\alpha_{n-1}<\beta_n<0.
\end{equation}
\item The polynomial $A_{n+1,\ell}(t)+zA_{n,\ell}(t)$ is real stable.
For every $s\in\mathbb R$, $A_{n+1,\ell}+sA_{n,\ell}$ has simple
real zeros.  All are negative if $s>-1$; one is zero if $s=-1$;
and exactly one is positive if $s<-1$.
\end{enumerate}
\end{maintheorem}
Here a probability measure on subsets is \emph{strongly Rayleigh}
if its multiaffine generating polynomial is real stable~\cite{BBL2009}.
Stability in Theorem~\ref{main:eulerian}\textup{(i)} is a direct
application of Theorem~\ref{stab:general}.  Section~\ref{alg:section}
proves strict interlacing using intermediate interlacers and records
the standard gamma-positivity consequences of real-rootedness and
palindromicity.

In the next section, we focus on the case $\ell=2$. For a synchronised pair of trees, join two nonroot labels when they
are siblings in either tree.  This \emph{sibling graph} is a union
of paths and even cycles.  Call the pair \emph{reduced} when each
odd component $C$ has exactly $(|C|-1)/2$ left-child labels.
There is no restriction on the orientations of even components.

\begin{maintheorem}[Gamma coefficients for pairs]\label{main:gamma}
\label{fs:pair-action}\label{gamma:reduced-pairs}\label{gamma:recurrence}
\begin{enumerate}[label=\textup{(\roman*)}]
\item Flipping the left--right statuses of all labels in an odd
sibling component, simultaneously in both trees, defines commuting
involutions.  Each orbit has a unique reduced pair.  If this pair has
$k$ common left-child labels, the orbit enumerator is
\begin{equation}\label{fs:pair-orbit}
 \sum_{\mathbf S\in\mathcal O}t^{|L(\mathbf S)|}
 =t^k(1+t)^{n-1-2k}.
\end{equation}
Consequently, $\gamma_{n,2,k}$ counts reduced pairs with $k$ common
left-child labels, and $\gamma_{n,2,0}=1$.
\item With $\gamma_{1,2,0}=1$ and coefficients outside their support
equal to zero, for $n\ge2$ one has
\begin{equation}\label{gamma:recurrence-equation}
 \gamma_{n,2,k}=(k+1)^2\gamma_{n-1,2,k}
 +(n-2k)(n+2k+1)\gamma_{n-1,2,k-1}.
\end{equation}
\end{enumerate}
\end{maintheorem}
Section~\ref{sec:gamma-action} constructs the action and proves the
recurrence. We mention how the action recovers the ordinary
unary-tree action when $\ell=1$.  The gamma-positivity for $\ell=2$
was proved in~\cite{Shankar2026}. Our Theorem~\ref{main:gamma} here gives
the orbit interpretation of its coefficients.

A finite graded poset is \emph{PECK} if it is rank-symmetric,
rank-unimodal, and strongly Sperner.  The last condition says that
the largest union of $k$ antichains has the size of the $k$ largest
ranks, for every $k$.  A symmetric Boolean (respectively, chain)
decomposition partitions the poset into Boolean subposets (respectively,
saturated chains) whose endpoint ranks sum to its rank.

\begin{maintheorem}[PECKness and symmetric decompositions]\label{main:peck}
\label{PECK:r1}\label{PECK:symmetric-decompositions}
For every $n\ge1$, $\mathcal Q_{n,1}$ admits a symmetric Boolean
decomposition and a symmetric chain decomposition.  On its tree basis,
let $E$ sum the moves of a unary right child to the left, let $F$ sum
the reverse moves, and set $HT=(2\rho(T)-n+1)T$.
These operators define an $\mathfrak{sl}_2$-representation, with $E$
supported on upward covers and $F$ on downward covers.
In particular, $\mathcal Q_{n,1}$ is PECK.
\end{maintheorem}
The operators and decompositions are proved in Section~\ref{sec:PECK}.
Palindromicity and Theorem~\ref{main:eulerian}\textup{(i)} give rank
symmetry and unimodality for all $\ell$.  We propose the remaining
order-theoretic conjecture.

\begin{mainconjecture}\label{PECK:conjecture}
For every $n\ge1$ and integer $\ell>1$, the poset $\mathcal Q_{n,\ell}$
is PECK.
\end{mainconjecture}

Finally, an $m$-Stirling permutation is a permutation of
$\{1^m,\ldots,n^m\}$ with no smaller letter between consecutive
occurrences of a letter.  Let $d(\sigma)$ count the number of internal descents.
Its descent history records whether this number increases when the
block of copies of $i$ is inserted, for $i=2,\ldots,n$.
Let $F^{(m)}_{n,\ell}$ enumerate tuples with the same history, with
$z_i$ marking an increase at step $i$, and let
$S^{(m)}_{n,\ell}(t)=F^{(m)}_{n,\ell}(t,\ldots,t)$.
Use maps $f:[n]\to[m(n-1)+1]$ satisfying $f(i)\le m(i-1)+1$
to define the synchronised kernel-refinement poset
$\mathcal R^{(m)}_{n,\ell}$.

\begin{maintheorem}[Stirling extensions]\label{main:stirling}
\label{stirling:stability}\label{stirling:tree-bijection}
For positive integers $m,\ell,n$, the following assertions hold.
\begin{enumerate}[label=\textup{(\roman*)}]
\item $F^{(m)}_{n,\ell}$ is real stable and its history distribution
is strongly Rayleigh.  For fixed $m,\ell$, the polynomials
$S^{(m)}_{n,\ell}$ have simple negative zeros and consecutive members
strictly interlace.  Their coefficients are log-concave and unimodal.
The row pencil $S^{(m)}_{n+1,\ell}(t)+zS^{(m)}_{n,\ell}(t)$ is real
stable; its real combinations have the same zero-sign alternatives
as in Theorem~\ref{main:eulerian}\textup{(iii)}.
\item The bounded functions above are in bijection with increasing
$(m+1)$-ary trees, with kernel blocks corresponding to maximal chains
in the rightmost slot.  Under the Gessel correspondence, non-root
block leaders are the insertion labels that increase the descent
number.  The poset $\mathcal R^{(m)}_{n,\ell}$ is graded of rank
$n-1$, with rank polynomial $S^{(m)}_{n,\ell}$.
\end{enumerate}
\end{maintheorem}
The stability assertion is another application of
Theorem~\ref{stab:general}, recorded in Table~\ref{stab:applications}.
Section~\ref{sec:stirling} proves the remaining assertions and discusses
variable multiplicities.  These polynomials record insertion histories,
rather than the final descent-top sets used in
\cite{haglund-visontai-stable-multivariate}.

For ordinary Stirling permutations ($m=2$), let $\mathcal R_{n,p}$
be the induced subposet of $\mathcal R^{(2)}_{n,1}$ with exactly $p$
plateaux, where a plateau is an adjacent equal pair.
In the ternary-tree model, fix all parent relations, middle edges,
and the order at doubly occupied outer slots; forget the orientations
at singly occupied outer slots.  Call this an \emph{outer-slot skeleton}.

\begin{maintheorem}[Fixed-plateau posets]\label{stirling:plateau-peck}
For $1\le p\le n$, the poset $\mathcal R_{n,p}$ is graded of rank
$p-1$, with rank function $\rho_p(\sigma)=d(\sigma)-(n-p)$.
It admits symmetric Boolean and chain decompositions and
cover-supported $\mathfrak{sl}_2$ raising and lowering operators.
In particular, it is PECK.  Its rank polynomial is
\begin{equation}\label{stirling:partial-gamma}
 t^{-(n-p)}\sum_{\operatorname{plat}(\sigma)=p}t^{d(\sigma)}
 =\sum_{b=0}^{\lfloor(p-1)/2\rfloor}
 g_{n,p,b}\,t^b(1+t)^{p-1-2b},
\end{equation}
where $g_{n,p,b}$ counts outer-slot skeletons with $b$ doubly occupied
outer pairs.
\end{maintheorem}
The partial gamma expansion and its tree action are known
\cite{MaMaYeh2019,ChenFuYan2023}.  The final part of
Section~\ref{sec:stirling} proves their compatibility with kernel
refinement, which gives the PECK and decomposition assertions.

\renewcommand{\theequation}{\thesection.\arabic{equation}}
\section{Subexceedant functions, trees, and posets}\label{sec:preliminaries}
We prove Theorem~\ref{main:models} by relating subexceedant functions,
increasing binary trees, and kernel-refinement posets.
We first recall the definition of the parametrised Eulerian polynomials.
We write $A_{\ell}(n,k)$ for the family denoted by $A^{(\ell)}_1(n,k)$
in \cite{Shankar2026}.  These parametrised Eulerian numbers satisfy the recurrence
\cite[Proposition~18]{Shankar2026}
\begin{equation}\label{prelim:recurrence}
 A_{\ell}(n,k)=(k+1)^{\ell}A_{\ell}(n-1,k)+(n-k)^{\ell}A_{\ell}(n-1,k-1)
 \qquad(n\geq2),
\end{equation}
with $A_{\ell}(1,0)=1$ and $A_{\ell}(n,k)=0$ for $k\notin\{0,\ldots,n-1\}$.
We write
\[
 A_{n,\ell}(t)=\sum_{k=0}^{n-1}A_{\ell}(n,k)t^k.
\]
When $\ell=1$, these are the ordinary Eulerian polynomials, with the convention
that the exponent records the number of descents.
The first nonconstant rows are
\begin{align*}
 A_{2,\ell}(t)&=1+t,\\
 A_{3,\ell}(t)&=1+2^{\ell+1}t+t^2,\\
 A_{4,\ell}(t)&=1+(2\cdot4^{\ell}+3^{\ell})t+(2\cdot4^{\ell}+3^{\ell})t^2+t^3.
\end{align*}
\subsection{Subexceedant functions and parametrised Eulerian numbers}

We use the subexceedant-function model introduced by Shankar
\cite[p.~5 and Section~4]{Shankar2026}. A \emph{subexceedant function} on $[n]$ is a map $f:[n]\to[n]$ satisfying
$f(i)\leq i$ for every $i$.  We also write it as the word $f(1)\cdots f(n)$.
Its \emph{kernel} $\ker(f)$ is the partition of $[n]$ into nonempty sets of
positions having the same image under $f$.  The \emph{block-leader set} is
\[
 \bl(f)=\{i\in[n]:f(i)\notin f([i-1])\},
\]
where $f([0])=\varnothing$.  Therefore, $\bl(f)$ is the set of
minima of the kernel blocks.
Let
\[
 \mathcal Q_{n,\ell}
 =\{(f_1,\ldots,f_\ell): f_a\text{ is subexceedant and }
                  \bl(f_1)=\cdots=\bl(f_\ell)\}.
\]
We call such tuples \emph{synchronised}.  For $\mathbf f\in\mathcal Q_{n,\ell}$,
write $B(\mathbf f)$ for its common block-minimum set.  We identify
$\mathcal Q_{n,1}$ with the set of subexceedant functions on $[n]$.

Recall Shankar's enumeration~\cite{Shankar2026}, whose proof also
identifies the two extension weights used below.
\begin{proposition}\label{prelim:tuple-enumeration}
The number of tuples $\mathbf f\in\mathcal Q_{n,\ell}$ with
$|B(\mathbf f)|=k+1$ is $A_{\ell}(n,k)$.
\end{proposition}
\begin{proof}
We begin by restricting a tuple to $[n-1]$.  If $n$ is not a new block minimum, each
coordinate repeats one of its $k+1$ existing values, giving $(k+1)^{\ell}$
extensions.  If $n$ is a new block minimum, the restricted tuple has $k$
blocks.  Each coordinate chooses one of the $n-k$ values in $[n]$ not
previously used, giving $(n-k)^{\ell}$ extensions.  The resulting recurrence
is \eqref{prelim:recurrence}, with the same initial condition.
\end{proof}

Encode a set $B\subseteq[n]$ containing $1$ by the binary word
$e=(e_2,\ldots,e_n)$, where $e_i=1$ if and only if $i\in B$.
Write $|e|=\sum_{i=2}^n e_i$, and define
\begin{equation}\label{prelim:word-count}
 c(e)=\prod_{i=2}^n
 \left(1+\bigl|\{j:2\leq j<i,\ e_j\ne e_i\}\bigr|\right),
\end{equation}
with the empty product equal to $1$.

The following is a combinatorial proof of palindromicity of the parametrised Eulerian polynomials.
\begin{proposition}\label{prelim:palindromicity}
For every binary word $e$, the number of subexceedant functions whose
block-minimum set is $\{1\}\cup\{i:e_i=1\}$ equals $c(e)$.  Consequently,
\begin{equation}\label{prelim:word-enumerator}
 A_{n,\ell}(t)=\sum_{e\in\{0,1\}^{n-1}}c(e)^{\ell} t^{|e|}.
\end{equation}
The polynomial $A_{n,\ell}$ is monic and palindromic of degree $n-1$, with
constant coefficient $1$ and positive coefficients in every degree.
\end{proposition}
\begin{proof}
Suppose that the values through position $i-1$ have been chosen, and put
$d_i=1+\sum_{j=2}^{i-1}e_j$.  If $e_i=0$, there are $d_i$ previously used
values available.  If $e_i=1$, there are $i-d_i$ unused values in $[i]$.
Both numbers equal the corresponding factor in
\eqref{prelim:word-count}.  Multiplying these choices proves the count and
independence of the $\ell$ coordinates proves \eqref{prelim:word-enumerator}.

Complementing all bits leaves $c(e)$ unchanged and replaces $|e|$ by
$n-1-|e|$, proving palindromicity.  The constant words have weight $1$,
so the constant and leading coefficients are $1$.  Every word has
positive weight, and words occur in every possible weight $|e|$.
\end{proof}

We retain the individual block-minimum labels in the multiaffine refinement
\begin{equation}\label{prelim:multivariate}
 F_{n,\ell}(z_2,\ldots,z_n)
 =\sum_{e\in\{0,1\}^{n-1}}c(e)^{\ell}\prod_{i:e_i=1}z_i.
\end{equation}
Its diagonal specialisation is $A_{n,\ell}(t)$.

\subsection{A kernel-preserving tree interpretation}\label{sec:trees}

An \emph{increasing plane binary tree} on $[n]$ is a rooted tree whose
vertices are labelled by $[n]$, whose edges increase away from the root,
and in which every vertex has at most one left child and at most one
right child.  Its root has label $1$.  Write $\mathcal T_n$ for
the set of these trees, and let $L(T)\subseteq\{2,\ldots,n\}$ be the set
of labels of left children in $T$.  A \emph{maximal right chain} is a
connected component of the graph obtained by deleting all left edges.
Its labels occur in increasing order from its root.

An ordered tuple $\mathbf T=(T_1,\ldots,T_\ell)\in\mathcal T_n^{\ell}$ is
said to be \emph{synchronised} if
\[
 L(T_1)=\cdots=L(T_\ell).
\]
Let $\mathcal T_{n,\ell}$ be the set of such tuples and write $L(\mathbf T)$
for their common left-child set.

The bijection in Theorem~\ref{main:models}\textup{(i)} preserves the entire kernel partition.  This
additional property will allow us to interpret tree moves as relations
in the posets considered below.

\begin{proof}[Proof of Theorem~\ref{main:models}\textup{(i)}]
We start by fixing a partition $\pi=\{B_1,\ldots,B_s\}$ of $[n]$, with its blocks
ordered so that
\[
 1=b_1<b_2<\cdots<b_s,\qquad b_j=\min B_j.
\]
A subexceedant function with kernel $\pi$ is specified by distinct
values $a_1,\ldots,a_s$, where $a_j\in[b_j]$, by setting $f(i)=a_j$
for $i\in B_j$. $a_1=1$ is forced. After $a_1,\ldots,a_{j-1}$
have been chosen, the available values for $a_j$ form the ordered set
\[
 I_j=[b_j]\setminus\{a_1,\ldots,a_{j-1}\},
 \qquad |I_j|=b_j-j+1.
\]

We form a right chain from each block $B_j$, listing its vertices in
increasing order.  For $j=2,\ldots,s$, we attach the root $b_j$ of the
$j$th chain as a left child of a vertex $p_j<b_j$ whose left-child
position is still unused.  The available parent labels form the
ordered set
\[
 J_j=[b_j-1]\setminus\{p_2,\ldots,p_{j-1}\},
 \qquad |J_j|=b_j-j+1.
\]
We choose $p_j$ to have the same position in $J_j$ as $a_j$ has in $I_j$,
where both sets are ordered increasingly.  All newly added edges can only
increase labels, every vertex other than $1$ has exactly one parent,
and no left-child position is used twice. Thus the result is an
increasing plane binary tree.  Its right chains are precisely the
chains prescribed by $\pi$.

In the other direction, the maximal right chains of a tree recover $\pi$ and their nonroot chain roots recover the parent labels $p_2,\ldots,p_s$.
Starting with $a_1=1$, we recover $a_j$ from the position of $p_j$ in
$J_j$ by choosing the value in the same position in $I_j$.  This
inverts the construction.  The roots of the nonroot right chains are
exactly the left children, proving the assertion about $L(\Phi_n(f))$. The coordinatewise extension and \eqref{tree:enumerator} follow from this.
\end{proof}

\begin{example}\label{ex:tree}
Consider $f=113223\in\mathcal Q_{6,1}$.  Its kernel blocks, in order of
their minima, are
\[
 B_1=\{1,2\},\qquad B_2=\{3,6\},\qquad B_3=\{4,5\},
\]
and their values are $(a_1,a_2,a_3)=(1,3,2)$.  For the root $b_2=3$,
the available values and parents are $I_2=\{2,3\}$ and $J_2=\{1,2\}$;
thus $a_2=3$ selects $p_2=2$.  For $b_3=4$, they are
$I_3=\{2,4\}$ and $J_3=\{1,3\}$, so $a_3=2$ selects $p_3=1$.
The resulting tree is shown in Figure~\ref{fig:tree-example}.
Its left-child set is $\{3,4\}=\bl(f)\setminus\{1\}$, and its maximal
right chains recover the three kernel blocks.
\end{example}
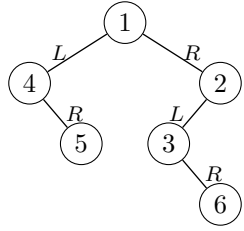
\begin{figure}[htbp]
\centering
\begin{tikzpicture}[x=1.05cm,y=1cm]
\node[treevertex] (a) at (0,0) {1};
\node[treevertex] (b) at (-1.2,-.8) {4};
\node[treevertex] (c) at (1.2,-.8) {2};
\node[treevertex] (d) at (-.55,-1.6) {5};
\node[treevertex] (e) at (.55,-1.6) {3};
\node[treevertex] (f) at (1.2,-2.4) {6};
\draw[treeedge] (a)--node[left,font=\scriptsize] {$L$}(b);
\draw[treeedge] (a)--node[right,font=\scriptsize] {$R$}(c);
\draw[treeedge] (b)--node[right,font=\scriptsize] {$R$}(d);
\draw[treeedge] (c)--node[left,font=\scriptsize] {$L$}(e);
\draw[treeedge] (e)--node[right,font=\scriptsize] {$R$}(f);
\end{tikzpicture}
\caption{The increasing binary tree $\Phi_6(113223)$.  Edge labels indicate
left and right child positions.}
\label{fig:tree-example}
\end{figure}

\subsection{The associated poset}\label{sec:posets}

We now order the synchronised tuples by simultaneous refinement of their
kernel partitions.  This order retains the subexceedant labellings and uses
the common block-minimum set as a rank parameter.

\begin{definition}\label{poset:definition}
For $\mathbf f,\mathbf g\in\mathcal Q_{n,\ell}$, define the strict order by
\[
 \mathbf f< \mathbf g
 \quad\Longleftrightarrow\quad
 \ker(f_a)\text{ properly coarsens }\ker(g_a)
 \quad(a\in[\ell]).
\]
Adjoining equality gives the order $\le$.  We use
$\mathcal Q_{n,\ell}$ for both the poset and its underlying set.
\end{definition}

Distinct tuples with the same tuple of kernels are incomparable.  In
particular, the order retains the subexceedant labellings of the blocks.

\begin{proof}[Proof of Theorem~\ref{main:models}\textup{(ii)}]
The defining relation is a partial order because proper refinement of set
partitions is transitive.  If $\mathbf f<\mathbf g$, then
$B(\mathbf f)\subsetneq B(\mathbf g)$: the minimum of each coarse block is
the minimum of one of the fine blocks it contains.

Suppose that $|B(\mathbf g)|-|B(\mathbf f)|\geq2$.  We choose
$v\in B(\mathbf g)\setminus B(\mathbf f)$.  In coordinate $a$, let $D_a$
be the block of $\ker(g_a)$ with minimum $v$, and let $C_a$ be the block
of $\ker(f_a)$ containing it.  We split $C_a$ into $D_a$ and
$C_a\setminus D_a$.  The latter is nonempty and retains the minimum of
$C_a$, since $v\notin B(\mathbf f)$.  The resulting partition has
block-minimum set $B(\mathbf f)\cup\{v\}$ and lies strictly between the
two original partitions.  Labelling each of its blocks by its minimum
gives a subexceedant function.  Performing this operation in every
coordinate therefore gives a tuple strictly between $\mathbf f$ and
$\mathbf g$, of rank one greater than $\mathbf f$.

Covers are exactly the comparable pairs whose numbers of
blocks differ by one.  The one-block partition has the unique
subexceedant labelling $11\cdots1$, and the discrete partition has the
unique subexceedant labelling $12\cdots n$.  This proves the assertions
about rank and bounds.  The rank enumeration follows from the
enumeration of synchronised tuples by their number of blocks.
\end{proof}

\section{A stability theorem for triangular recurrences}\label{stab:section}
We prove Theorem~\ref{stab:general}.  After recalling stability
preservers, we develop the path normalisation and finite multiplier
lemmas used in the proof.

\subsection{Stability preservers}

A nonzero polynomial with real coefficients is \emph{real stable} if it
does not vanish when all its variables have positive imaginary part.
For a univariate real polynomial, stability is equivalent to real-rootedness.
We recall the standard stability preservers; see
Srivastava~\cite[Theorem~2.7]{Srivastava2015}.

\begin{proposition}[Stability preservers]\label{stab:standard-preservers}
The following operations preserve real stability, provided the resulting
polynomial is nonzero:
\begin{enumerate}
\item multiplication of real stable polynomials;
\item permutation of variables and substitution $z_i\mapsto a_i z_i$
      with $a_i>0$;
\item identification of two or more variables;
\item specialisation of any variable to a real constant;
\item partial differentiation;
\item the substitution
      $f(\mathbf z)\mapsto z_i^{d_i}
      f(z_1,\ldots,-1/z_i,\ldots,z_m)$, where
      $d_i=\deg_{z_i}f$.
\end{enumerate}
Only specialisation and differentiation can produce the zero polynomial
from a nonzero stable input in this list.
\end{proposition}

We also use the homogenization theorem of Borcea, Br\"and\'en and
Liggett~\cite[Theorem~4.5]{BBL2009}: a polynomial with nonnegative
coefficients is real stable if and only if its homogenization is real stable.
In particular, for every integer $D\ge\deg f$,
\begin{equation}\label{stab:homogenization}
 H_D(\mathbf z,y)=y^D f(\mathbf z/y)
\end{equation}
is real stable whenever $f$ is real stable and has nonnegative coefficients.
Polarization also preserves stability
\cite[Corollary~4.7]{BBL2009}: replacing $y^j$ by
$e_j(y_1,\ldots,y_D)/\binom Dj$, where $e_j$ is the elementary
symmetric polynomial, preserves stability for polynomials of degree at most
$D$ in $y$.

\subsection{Weighted paths and multivariate refinements}

We use the path interpretation in \cite[Section~2 and Theorem~5]{Shankar2026}.
Fix an initial vertex $(k_0,n_0)$, and consider paths whose steps are
$N=(0,1)$ and $C=(1,1)$, called north and cross steps, respectively.
For the affine triangular recurrence
\begin{equation}\label{stab:affine-recurrence}
 T_\ell(n,k)=c(n,k)^\ell T_\ell(n-1,k)
              +d(n,k)^\ell T_\ell(n-1,k-1),
\end{equation}
where
\[
 c(n,k)=\alpha n+\beta k+\gamma,
 \qquad d(n,k)=\alpha'n+\beta'k+\gamma',
\]
we take $T_\ell(n_0,k_0)=1$ and zero entries outside
$k_0\le k\le k_0+n-n_0$.
A north step ending at $(k,i)$ has weight $c(i,k)$, and a cross step
ending at $(k,i)$ has weight $d(i,k)$.  Write $\wt(p)$ for the product of
these weights along a path $p$, and $C(p)$ for the set of heights $i$
at which $p$ takes a cross step.

Let $\mathcal P_n$ be the set of all paths from $(k_0,n_0)$ to height $n$.
Define
\begin{equation}\label{stab:path-polynomial}
 \mathcal F_{n,\ell}(z_{n_0+1},\ldots,z_n)
 =\sum_{p\in\mathcal P_n}\wt(p)^\ell\prod_{i\in C(p)}z_i,
 \qquad \mathcal F_{n_0,\ell}=1.
\end{equation}
Equivalently, a cross step ending at height $i$ receives the additional
weight $z_i$, while every scalar step weight is raised to the power $\ell$.
The path decomposition according to the last step gives
\[
 \mathcal F_{n,\ell}(t,\ldots,t)
 =\sum_{j=0}^{n-n_0}T_\ell(n,k_0+j)t^j.
\]
Thus the variables record the cross-step heights. For a prefix with $j$ cross steps, put
\begin{equation}\label{stab:outgoing-weights}
 a_n(j)=c(n,k_0+j),\qquad b_n(j)=d(n,k_0+j+1).
\end{equation}
These are the weights of its two outgoing steps. The following lemma, for which we do not include the proof, allows us the write the last step decomposition in terms of the operator $\Theta=\sum_{i=n_0+1}^{n-1}z_i\partial_{z_i}$.
\begin{lemma}\label{stab: rational}
    For $f=\sum_{j=0}^d f_j$ with $f_j$ homogeneous of degree $j$ in the variables $z_{n_0+1},\dots, z_{n-1}$, and $\Theta=\sum_{i=n_0+1}^{n-1}z_i\partial_{z_i}$, we get
    \begin{equation}
        \Theta(f)=\sum_{j=0}^d jf_j
    \end{equation}
    Equivalently, for any rational function $\lambda(j)$ in $j$ where the denominator does not vanish at $j\in \{0,1,\dots, d\}$, we have
    \begin{equation}
        \lambda(\Theta)=\sum_{j=0}^d \lambda(j)f_j.
    \end{equation}
\end{lemma}

So, the last-step decomposition
can therefore be written as
\begin{equation}\label{stab:path-recurrence}
 \mathcal F_{n,\ell}
 =\bigl[a_n(\Theta)^\ell+z_n b_n(\Theta)^\ell\bigr]
   \mathcal F_{n-1,\ell}.
\end{equation}
Notice the shift by one in $b_n(j)$ as a cross step increases the rank.

\subsection{Finite multiplier sequences}

The P\'olya--Schur theorem characterizes diagonal operators that preserve
real-rootedness \cite{PolyaSchur1914}.  We use its finite-degree form
\cite[Theorem~3.7]{CravenCsordas1977}; see also
\cite[Theorem~4]{KatkovaShapiroVishnyakova2011}.  The following formulation
for homogeneous components follows by homogenization and polarization.

\begin{lemma}\label{stab:finite-multiplier}
Let $\eta_0,\ldots,\eta_d\ge0$.  For $f=\sum_{j=0}^d f_j$, with $f_j$
homogeneous of degree $j$, set
\[
 M_\eta f=\sum_{j=0}^d\eta_j f_j,
 \qquad
 J_d(\eta;t)=\sum_{j=0}^d\binom dj\eta_jt^j.
\]
where $J_d$ is the Jensen polynomial associated with the sequence $\eta_j$. The operator $M_\eta$ preserves real stability, allowing the zero
polynomial, on polynomials with nonnegative coefficients and total degree
at most $d$, in any number of variables, if and only if
$J_d(\eta;t)$ is zero or has only nonpositive real zeros.
\end{lemma}

\begin{proof}
Necessity follows by applying the operator to $(1+t)^d$.
For sufficiency, let $f$ be real stable with nonnegative coefficients,
and first suppose that
$J_d(\eta;t)=a\prod_{i=1}^d(t+\rho_i)$, where $a>0$ and $\rho_i\ge0$.
We homogenize $f$ to degree $d$ and then polarize the homogenizing variable.
The resulting stable polynomial is
\[
 \widetilde H(\mathbf z,y_1,\ldots,y_d)
 =\sum_{j=0}^d f_j(\mathbf z)
   \frac{e_{d-j}(y_1,\ldots,y_d)}{\binom dj}.
\]
Comparing coefficients in the factorization of $J_d$ gives
\[
 M_\eta f=a\widetilde H(\mathbf z,\rho_1,\ldots,\rho_d),
\]
which is stable or zero by real specialisation.  If $J_d$ has degree
$q<d$, we can apply this argument to
$J_d(\eta;t)(1+\varepsilon t)^{d-q}$ and let $\varepsilon$ tend to zero from above.
Coefficient limits of stable polynomials are stable or zero, by Hurwitz's
theorem.  The zero polynomial case is immediate.
\end{proof}

\begin{lemma}\label{stab:euler-operators}
Let $f$ be a nonzero real stable polynomial with nonnegative coefficients,
put $\Theta=\sum_i z_i\partial_{z_i}$, and let $D\ge\deg f$ be real.
For every $a>0$, the polynomial $(\Theta+a)f$ is real stable and has
nonnegative coefficients.  The same conclusion holds for $(D-\Theta)f$
unless this polynomial is zero.
\end{lemma}

\begin{proof}
Writing $d=\deg f$, the degree multipliers are $j+a$ and $D-j$.
For $d\ge1$, their Jensen polynomials are
\[
 (1+t)^{d-1}\bigl(a+(a+d)t\bigr),\qquad
 (1+t)^{d-1}\bigl(D+(D-d)t\bigr),
\]
respectively.  We can now apply Lemma \ref{stab:finite-multiplier}; the case $d=0$ is
immediate.  The first operator has strictly positive degree multipliers
and hence cannot annihilate $f$.
\end{proof}

For $A>0$, write $(A)_j=A(A+1)\cdots(A+j-1)$ and
$C^{\underline j}=C(C-1)\cdots(C-j+1)$, with both products equal to one
when $j=0$.

\begin{lemma}\label{stab:hypergeometric}
If $A>0$, $C>d-1$ and $q>0$, the sequences
\[
 q^j\frac{C^{\underline j}}{(A)_j},\qquad
 \frac{q^j}{(A)_j},\qquad q^jC^{\underline j}
 \qquad(0\le j\le d)
\]
are finite multiplier sequences in the sense of
Lemma \ref{stab:finite-multiplier}.
\end{lemma}

\begin{proof}
The Jensen polynomial for the first sequence is the terminating
hypergeometric polynomial
\begin{equation}\label{stab:hypergeometric-symbol}
 \sum_{j=0}^d\binom dj\frac{C^{\underline j}}{(A)_j}(qt)^j
 ={}_2F_1(-d,-C;A;qt).
\end{equation}
Driver and Jordaan's classification
\cite[Theorem~3.2(iv),(v)]{DriverJordaan2002} gives its negative real zeros
in this parameter range.  In this case the location can also be seen
directly from the Pfaff transformation and the Jacobi representation:
\begin{equation}\label{stab:jacobi-symbol}
 {}_2F_1(-d,-C;A;qt)
 =\frac{d!}{(A)_d}(1-qt)^d
   P_d^{(A-1,C-d)}\!\left(\frac{1+qt}{1-qt}\right).
\end{equation}
Both Jacobi parameters exceed $-1$, so all its zeros $x_i$ lie in
$(-1,1)$.  They give the $d$ negative zeros
$t_i=(x_i-1)/(q(x_i+1))$ of \eqref{stab:hypergeometric-symbol}.
This also includes the boundary $C=d$ between the two cases in the cited
classification.

For the second sequence, we let $C\to\infty$ in the first one after
replacing $q$ by $q/C$.  For the third, we let $A\to\infty$ after replacing
$q$ by $qA$.  The Jensen polynomials converge coefficientwise to
polynomials of degree $d$ with positive coefficients and only negative
real zeros.  We can now apply the finite multiplier criterion to complete the proof.
\end{proof}

\begin{lemma}\label{stab:binomial-operator}
For a polynomial $f=\sum_{j=0}^m f_j$, where $f_j$ is homogeneous of
degree $j$, define
\[
 \mathcal B_m f=\sum_{j=0}^m\binom mj f_j.
\]
If $f$ is nonzero, real stable, and has nonnegative coefficients, then
$\mathcal B_m f$ has the same properties.
\end{lemma}

\begin{proof}
We take $A=1$, $C=m$ and $q=1$ in Lemma \ref{stab:hypergeometric}.
The resulting multiplier is $m^{\underline j}/(1)_j=\binom mj$.
\end{proof}

\subsection{The affine stability theorem}

We now impose hypotheses on \eqref{stab:affine-recurrence}.  Assume
\begin{equation}\label{stab:affine-signs}
 \alpha,\alpha',\beta\ge0,\qquad -\alpha'\le\beta'\le0,
\end{equation}
as in \cite[Theorem~4]{Shankar2026}, and require strictly positive weights
on all accessible outgoing steps:
\begin{equation}\label{stab:accessible-positivity}
 a_n(j)>0,\quad b_n(j)>0
 \qquad(n>n_0,\ 0\le j\le n-n_0-1).
\end{equation}
The additional arithmetic conditions for our stability theorem are
\begin{equation}\label{stab:integer-increments}
 \frac{\alpha}{\beta}\in\mathbb Z_{\ge0}\quad\text{if }\beta>0,
 \qquad
 \frac{\alpha'}{-\beta'}\in\mathbb Z_{\ge1}\quad\text{if }\beta'<0.
\end{equation}
No ratio is required when its denominator is zero.

For $0\le j\le n-n_0$, define
\begin{equation}\label{stab:general-normalization}
 w_n(0)=1,\qquad
 w_n(j)=\prod_{i=0}^{j-1}\frac{b_n(i)}{a_n(i)},\qquad
 \mathcal P_{n,\ell}
 =\sum_{p\in\mathcal P_n}
   \frac{\wt(p)^\ell}{w_n(|C(p)|)^\ell}\prod_{i\in C(p)}z_i.
\end{equation}
For $n=n_0$, set $\mathcal P_{n_0,\ell}=1$.

\begin{proof}[Proof of Theorem~\ref{stab:general}]
We set $m=n-n_0$.  If $\beta>0$ and $\beta'<0$, we write
\[
 a_n(j)=u_n(j+A_n),\qquad b_n(j)=v_n(C_n-j),
\]
where
\[
 u_n=\beta,\quad
 A_n=\frac{\alpha n+\beta k_0+\gamma}{\beta},\qquad
 v_n=-\beta',\quad
 C_n=\frac{\alpha'n+\beta'(k_0+1)+\gamma'}{-\beta'}.
\]
Then $A_n>0$, $C_n>m-1$, and
$A_n-A_{n-1}=s:=\alpha/\beta$ and
$C_n-C_{n-1}=h:=\alpha'/(-\beta')$ are nonnegative integers.
If $\beta=0$, we instead take $u_n=\alpha n+\gamma$ and omit the factor
$j+A_n$ in all formulas below.  If $\beta'=0$, we take
$v_n=\alpha'n+\gamma'$ and omit the factor $C_n-j$.
The corresponding factors involving $A_n$ or $C_n$ are likewise omitted.
With this convention, the normalisation becomes
\begin{equation}\label{stab:affine-weights}
 w_n(j)=\left(\frac{v_n}{u_n}\right)^j
          \frac{C_n^{\underline j}}{(A_n)_j}.
\end{equation}

A path ending at height $n-1$ with $j$ cross steps acquires normalised weight (without raising to the power $\ell$)
\[
 a_n(j)\frac{w_{n-1}(j)}{w_n(j)}
 =b_n(j)\frac{w_{n-1}(j)}{w_n(j+1)}
 =:\lambda_n(j)
\]
under either choice of last step. Therefore, by Lemma \ref{stab: rational},
\begin{equation}\label{stab:balanced-recurrence}
 \mathcal P_{n,\ell}
 =(1+z_n)\lambda_n(\Theta)^\ell\mathcal P_{n-1,\ell}.
\end{equation}
Here a function of $\Theta$ acts by its value at $j$ on the homogeneous
component of degree $j$.

The first step gives
$\mathcal P_{n_0+1,\ell}=a_{n_0+1}(0)^\ell(1+z_{n_0+1})$.
For $n\ge n_0+2$, we put
$\xi_n=v_{n-1}u_n/(u_{n-1}v_n)>0$.
Cancellation of rising and falling factorials in
\eqref{stab:affine-weights} gives
\begin{equation}\label{stab:factored-transition}
 \lambda_n(j)
 =u_n\xi_n^j
   \frac{(A_{n-1}+j)_{s+1}}{(A_{n-1})_s}
   \frac{(C_{n-1}-j+1)_h}{(C_{n-1}+1)_h}.
\end{equation}
The factor $\xi_n^j$ rescales all old variables by the positive number
$\xi_n$.  Every numerator factor involving $A_{n-1}+j$ corresponds
to $\Theta+a$ with $a>0$.  Every numerator factor involving
$C_{n-1}-j+1$ corresponds to $D-\Theta$, with
$D\ge C_{n-1}+1>m-1$.
Since $\deg\mathcal P_{n-1,\ell}=m-1$, all these operators preserve
nonnegative real stability by Lemma \ref{stab:euler-operators}.
They have strictly positive multipliers on the degrees in question.
Their $\ell$-fold compositions, followed by multiplication by $1+z_n$,
prove the stability of $\mathcal P_{n,\ell}$ inductively.

Finally, \eqref{stab:affine-weights} is one of the finite multiplier
sequences in Lemma \ref{stab:hypergeometric}; if both rank slopes vanish,
it is a positive geometric sequence and simply rescales the variables.
Restoring the original homogeneous components gives
\[
 \mathcal F_{n,\ell}=M_{w_n}^{\,\ell}\mathcal P_{n,\ell}.
\]
This proves its stability.  Diagonal specialisation proves real-rootedness
of the shifted row polynomial.  Its positive coefficients and positive
constant coefficient exclude nonnegative zeros.
\end{proof}

Table~\ref{stab:applications} lists coefficient pairs covered by
Theorem~\ref{stab:general}, selected from
\cite[Table~1]{Shankar2026}, together with the parametrised Eulerian
specialisation.  The names in the first column refer to the classical
$\ell=1$ recurrences; the stated conclusions hold after raising both
step weights to every positive integer power $\ell$.
The displayed initial vertices ensure positive accessible weights.
The Stirling cycle, Stirling subset, and Lah families also admit the
$(r,r)$ initialization, $r\ge1$, used for their $(\ell,r)$-variants
in~\cite{Shankar2026}.  In the row labelled $S(r,n,k)$, the
initial vertex $(1,1)$ corresponds to the reindexing
$T(n,k)=S_r(n+r-1,k+r-1)$ of the ordinary $r$-Stirling subset numbers;
see also~\cite[Remark~16]{Shankar2026}.

\begin{table}[htbp]
\centering
\small
\setlength{\tabcolsep}{4pt}
\renewcommand{\arraystretch}{1.19}
\begin{tabularx}{\textwidth}{@{}Xcccl@{}}
\toprule
Recurrence family & $c(n,k)$ & $d(n,k)$ & $(k_0,n_0)$
 & Range covered \\
\midrule
Parametrised Eulerian & $k+1$ & $n-k$ & $(0,1)$ & All $\ell\ge1$ \\
$m$-Stirling descents ($m$-order Eulerian) & $k+1$ & $m(n-1)-k+1$ & $(0,1)$
 & $m\in\mathbb Z_{\ge1}$ \\
$1/q$-Eulerian & $qk+1$ & $q(n-k)$ & $(0,1)$ & $q>0$ \\
Stirling subset & $k$ & $1$ & $(1,1)$ & All $\ell\ge1$ \\
Lah & $n+k-1$ & $1$ & $(1,1)$ & All $\ell\ge1$ \\
Generalised Lah--Stirling & $2n+k-2$ & $1$ & $(1,1)$
 & All $\ell\ge1$ \\
$S(r,n,k)$ & $k+r-1$ & $1$ & $(1,1)$ & $r\in\mathbb Z_{\ge1}$ \\
Binomial coefficients & $1$ & $1$ & $(0,0)$ & All $\ell\ge1$ \\
Stirling cycle & $n-1$ & $1$ & $(1,1)$ & All $\ell\ge1$ \\
Holiday, first kind & $2n+k-1$ & $1$ & $(0,0)$
 & All $\ell\ge1$ \\
Holiday, second kind & $2n+k$ & $1$ & $(0,0)$
 & All $\ell\ge1$ \\
\bottomrule
\end{tabularx}
\caption{Coefficient pairs and admissible initial vertices for the
powered recurrence.  Every row has stable cross-step
polynomials and normalised cross-step polynomials.  The shifted row
polynomials have only negative real zeros.  The $m$-order Eulerian
row is the $m$-Stirling descent-history recurrence of
Section~\ref{sec:stirling}.}
\label{stab:applications}
\end{table}

For the two specialisations studied below, the normalisation from
\eqref{stab:general-normalization} simplifies to
\[
\begin{array}{c|c|c}
\text{family}&(a_n(j),b_n(j))&w_n(j)\\ \hline
\text{parametrised Eulerian}&(j+1,n-1-j)&\displaystyle\binom{n-1}{j}\\[4pt]
m\text{-Stirling histories}&(j+1,m(n-1)-j)&\displaystyle\binom{m(n-1)}{j}
\end{array}
\]
Their multivariate stability statements therefore follow directly
from Theorem~\ref{stab:general}.

\begin{remark}\label{stab:scope}
The arithmetic conditions \eqref{stab:integer-increments} are used to
factor the transition operator in \eqref{stab:factored-transition}.
The theorem does not assert stability throughout the larger real
parameter range of the log-concavity theorem in \cite{Shankar2026}.
\end{remark}

\begin{remark}[General triangular arrays]\label{stab:general-arrays}
The balancing identity \eqref{stab:balanced-recurrence} does not require
affine coefficients.  For arbitrary positive outgoing weights
$a_n(j),b_n(j)$, use \eqref{stab:general-normalization} and
$\lambda_n(j)=a_n(j)w_{n-1}(j)/w_n(j)$.
The same proof gives stability for every positive integer $\ell$ whenever
both Jensen polynomials
\[
 \sum_{j=0}^{m-1}\binom{m-1}{j}\lambda_n(j)t^j,
 \qquad
 \sum_{j=0}^{m}\binom mj w_n(j)t^j
 \qquad(m=n-n_0)
\]
have only negative real zeros at every step, with constant polynomials
allowed.  Thus extending the result to a further triangular array requires
checking both the transition multiplier and the multiplier that restores
the original coefficients.
\end{remark}

\section{Parametrised Eulerian polynomials}\label{alg:section}
We prove Theorem~\ref{main:eulerian}.  Its stability assertion follows
directly from the Eulerian specialisation in Table~\ref{stab:applications}.

\subsection{Real stability and real-rootedness}

The binomial normalisation below is the multivariate form of
Alexandersson's univariate normalisation~\cite[Section~5]{Alexandersson2026}.
Theorem~\ref{main:eulerian} gives a stable refinement and strengthens
the consecutive interlacing to strict interlacing.
Recall that
\[
 F_{n,\ell}(z_2,\ldots,z_n)
 =\sum_{e\in\{0,1\}^{n-1}}c(e)^{\ell}\prod_{i:e_i=1}z_i,
 \qquad F_{1,\ell}=1.
\]
Here the entries of $e$ are indexed by $2,\ldots,n$. In particular,
$F_{n,\ell}(t,\ldots,t)=A_{n,\ell}(t)$.

We also introduce the normalised polynomial
\begin{equation}\label{stab:normalized-definition}
 P_{n,\ell}(z_2,\ldots,z_n)
 =\sum_{e\in\{0,1\}^{n-1}}
 \frac{c(e)^{\ell}}{\binom{n-1}{|e|}^{\ell}}
 \prod_{i:e_i=1}z_i.
\end{equation}
Its diagonal specialisation will be denoted by
\[
 \widehat A_{n,\ell}(t)
 =P_{n,\ell}(t,\ldots,t)
 =\sum_{k=0}^{n-1}\frac{A_{\ell}(n,k)}{\binom{n-1}{k}^{\ell}}t^k.
\]

\begin{proof}[Proof of Theorem~\ref{main:eulerian}\textup{(i)}]
The Eulerian row of Table~\ref{stab:applications} identifies $F_{n,\ell}$
and $P_{n,\ell}$ with the two polynomials of
Theorem~\ref{stab:general}, which proves stability.
Positive rescaling and diagonal specialisation give the univariate
assertions; their positive coefficients and constant terms exclude
nonnegative zeros.  Dividing $F_{n,\ell}$ by $F_{n,\ell}(1,\ldots,1)$
gives the probability generating polynomial of the common block-minimum
set, so this distribution is strongly Rayleigh.
\end{proof}

For the interlacing argument, we record the specialisations of
\eqref{stab:path-recurrence} and \eqref{stab:balanced-recurrence}:
\begin{align}
 F_{n,\ell}
 &=\left[(\Theta+1)^\ell+z_n(n-1-\Theta)^\ell\right]F_{n-1,\ell},
 \label{stab:unnormalized-recurrence}\\
 P_{n,\ell}
 &=\frac{1+z_n}{(n-1)^\ell}
   \bigl[(\Theta+1)(n-1-\Theta)\bigr]^\ell P_{n-1,\ell},
 \label{stab:normalized-recurrence}
\end{align}
where $\Theta=\sum_{i=2}^{n-1}z_i\partial_{z_i}$ and
$F_{1,\ell}=P_{1,\ell}=1$.

\subsection{Strict interlacing}\label{int:section}

We next compare consecutive rows for a fixed parameter. The argument uses
two intermediate polynomials: one shifts every root to the right and the
other shifts every root to the left. Write $\theta=t\,d/dt$.

\begin{lemma}\label{int:root-movement}
 Let $f$ have positive leading coefficient and simple negative zeros
 $\rho_1<\cdots<\rho_d<0$, where $d\geq1$, and let $a>0$.
 Then $(\theta+a)f$ has one simple zero in each of the intervals
\[
 (\rho_1,\rho_2),\ldots,(\rho_{d-1},\rho_d),(\rho_d,0),
\]
and $(d+a-\theta)f$ has one simple zero in each of
\[
 (-\infty,\rho_1),(\rho_1,\rho_2),\ldots,(\rho_{d-1},\rho_d).
\]
Both transformed polynomials have degree $d$ and positive leading
coefficient.
\end{lemma}

\begin{proof}
 At every zero of $f$,
\[
 ((\theta+a)f)(\rho_i)=\rho_i f'(\rho_i),
 \qquad
 ((d+a-\theta)f)(\rho_i)=-\rho_i f'(\rho_i).
\]
The signs alternate with $i$. The first polynomial is positive at zero,
and the second has leading coefficient equal to $a$ times that of $f$.
These endpoint signs give a zero in each indicated interval.
Each transformed polynomial has degree $d$, so these $d$ distinct zeros
account for all its zeros and are simple.
\end{proof}

For two real-rooted polynomials of the same degree, we say that their
zeros \emph{weakly interlace} if the ordered zeros alternate, allowing
equalities. We use the following Hermite-Biehler criterion
\cite[Theorem~4.11]{BBL2009}: if $g(t)+z h(t)$ is real stable and
$g,h$ are nonzero real polynomials, then their zeros interlace.

\begin{proof}[Proof of Theorem~\ref{main:eulerian}\textup{(ii)}]
 The initial polynomial is $A_{2,\ell}(t)=1+t$. Suppose that
 $f=A_{n,\ell}$ has simple negative zeros $\alpha_1<\cdots<\alpha_d$,
 where $d=n-1$. We set
\begin{equation}\label{int:intermediate-polynomials}
 U=(\theta+1)^{\ell} f,
 \qquad V=(n-\theta)^{\ell} f.
\end{equation}
The recurrence gives
\begin{equation}\label{int:row-recurrence}
 A_{n+1,\ell}(t)=U(t)+tV(t).
\end{equation}
Both $U$ and $V$ have degree $d$ and positive leading coefficient.
By repeated application of Lemma \ref{int:root-movement}, their zeros
$u_1<\cdots<u_d$ and $v_1<\cdots<v_d$ are simple and negative, and
\begin{equation}\label{int:coordinatewise}
 v_i<\alpha_i<u_i\qquad(1\leq i\leq d).
\end{equation}
Here we have iterated only the coordinatewise root inequalities: each
application of $\theta+1$ moves its $i$th zero to the right, and each
application of $n-\theta=d+1-\theta$ moves it to the left.

On the other hand,~\Cref{stab:main} and
\eqref{stab:unnormalized-recurrence} show that
\begin{equation}\label{int:stable-pencil}
 F_{n+1,\ell}(\underbrace{t,\ldots,t}_{n-1\text{ variables}},z)
 =U(t)+zV(t)
\end{equation}
is real stable. The preceding Hermite--Biehler criterion implies weak interlacing
of $U$ and $V$. The inequality $v_1<u_1$ fixes its orientation, giving
\begin{equation}\label{int:intermediate-order}
 v_1\leq u_1\leq v_2\leq u_2\leq\cdots\leq v_d\leq u_d.
\end{equation}
Combining \eqref{int:coordinatewise} and
\eqref{int:intermediate-order}, we obtain
\[
 u_{i-1}\leq v_i<\alpha_i<u_i\leq v_{i+1},
\]
with nonexistent endpoints omitted. Exactly $i-1$ zeros of $U$ and
exactly $i$ zeros of $V$ therefore lie to the left of $\alpha_i$.
Their leading coefficients are positive, so
\[
 \operatorname{sgn}U(\alpha_i)=(-1)^{d-i+1},
 \qquad
 \operatorname{sgn}V(\alpha_i)=(-1)^{d-i}.
\]
Since $\alpha_i<0$, the two terms in
$U(\alpha_i)+\alpha_iV(\alpha_i)$ have the same nonzero sign. Hence
\begin{equation}\label{int:signs-at-old-roots}
 \operatorname{sgn}A_{n+1,\ell}(\alpha_i)=(-1)^{d-i+1}.
\end{equation}
The constant and leading coefficients of $A_{n+1,\ell}$ are both $1$.
Together with \eqref{int:signs-at-old-roots}, the signs at $-\infty$
and at zero give one zero in each of
\[
 (-\infty,\alpha_1),\quad
 (\alpha_1,\alpha_2),\ldots,(\alpha_{d-1},\alpha_d),\quad
 (\alpha_d,0).
\]
These $d+1$ distinct zeros exhaust the degree. They are consequently
simple and satisfy \eqref{int:strict-chain}, completing the induction.
\end{proof}

\begin{proof}[Proof of Theorem~\ref{main:eulerian}\textup{(iii)}]
 The case $n=1$ is immediate. For $n\geq2$, we set
 $q=A_{n+1,\ell}$ and $f=A_{n,\ell}$, and write
 $\alpha_1<\cdots<\alpha_{n-1}$ for the zeros of $f$.
 Both are monic, and strict interlacing
 gives the partial-fraction expansion
\[
 \frac{q(t)}{f(t)}
 =t+b-\sum_{i=1}^{n-1}\frac{c_i}{t-\alpha_i},
 \qquad c_i=-\frac{q(\alpha_i)}{f'(\alpha_i)}>0.
\]
This ratio has positive imaginary part in the upper half-plane, so
$q(t)+zf(t)$ cannot vanish when both $t$ and $z$ are in that
half-plane. On each real interval between consecutive poles, including
the two unbounded intervals, its derivative is
\[
 1+\sum_{i=1}^{n-1}\frac{c_i}{(t-\alpha_i)^2}>0,
\]
and its range is all of $\mathbb R$. Thus $q/f=-s$ has precisely one
simple solution in each interval. Only the rightmost interval intersects
$[0,\infty)$, and $q(0)/f(0)=1$ determines the asserted sign of its
solution.
\end{proof}

\subsection{Gamma coefficients and the Strong Rayleigh property}
\label{sec:gamma-consequences}

Theorem~\ref{main:eulerian}\textup{(i)} recovers the known
gamma-positivity of $A_{n,\ell}$.  Indeed,
Proposition~\ref{prelim:palindromicity} gives palindromicity and
nonnegative coefficients, so Br\"and\'en's criterion
\cite[Remark~3.1]{Branden2015} gives nonnegative coefficients in
\eqref{alg:gamma-expansion}.  They are integers by successive
coefficient comparison, since $t^j(1+t)^{n-1-2j}$ has lowest term
$t^j$ with coefficient $1$.
The stronger real-rootedness of $\Gamma_{n,\ell}$ is included in
Alexandersson's theorem~\cite[Theorem~5.1]{Alexandersson2026}.
Newton's inequalities imply log-concavity and unimodality of the
coefficients of $A_{n,\ell}$; see \cite[Section~1]{Branden2015}.

In probabilistic terms, the stable polynomial $F_{n,\ell}$ gives the
strongly Rayleigh measure
\[
 \mu_{n,\ell}(S)=\frac{c(\mathbf1_S)^{\ell}}{A_{n,\ell}(1)}
 \qquad(S\subseteq\{2,\ldots,n\}).
\]
Here $\mathbf1_S$ is the binary word with ones in the positions of $S$.
This is the distribution of the common block-minimum set, with $1$
removed, under the uniform measure on $\mathcal Q_{n,\ell}$.

\section{Combinatorial interpretations of the gamma coefficients}\label{sec:gamma-action}
We prove Theorem~\ref{main:gamma} by extending unary flips on a single
increasing binary tree (the Foata-Strehl group action) to synchronised flips on pairs.

\subsection{Unary flips for a single tree}

We first recall the tree form of the modified Foata--Strehl action;
see \cite{FoataStrehl1974,Branden2008,Petersen2015}.  A vertex with
two children is \emph{binary}, and a vertex with exactly one child is
\emph{unary}.  For $v\in[n]$, define $\phi_v(T)$ by moving the unique
child subtree of $v$ to the opposite side when $v$ is unary, and by
fixing $T$ otherwise.  The parent relations do not change.  Therefore
the unary vertices remain the same, each $\phi_v$ is an involution,
and the maps $\phi_v$ commute.  They define an action of
$(\mathbb Z/2\mathbb Z)^n$ on $\mathcal T_n$.

\begin{proposition}\label{fs:single-tree}
An orbit whose trees have $b$ binary vertices and $u$ unary vertices
has enumerator
\[
 t^b(1+t)^u=t^b(1+t)^{n-1-2b}.
\]
Consequently, $\gamma_{n,1,k}$ counts increasing plane binary trees
with $k$ binary vertices and with every unary child on the right.
\end{proposition}

\begin{proof}
Within an orbit, the orientation of each unary edge can be chosen
independently, while the left and right children at binary vertices
are fixed.  Each binary vertex contributes one left child.  Each
unary edge contributes either zero or one left child, giving the
factor $1+t$.  Counting edges gives $2b+u=n-1$.  Every orbit contains
a unique tree in which all unary children are on the right, and its
number of left children is $b$.  Summing the orbit enumerators proves
the gamma-coefficient interpretation.
\end{proof}

\subsection{Sibling graphs and synchronised flips}

For $\mathbf T\in\mathcal T_{n,\ell}$, its \emph{sibling graph}
$G(\mathbf T)$ is the simple graph on $\{2,\ldots,n\}$ with an edge
$\{i,j\}$ whenever $i$ and $j$ are siblings in at least one coordinate
tree.  Each tree contributes a matching to this graph.  Siblings have
opposite left--right status, so $G(\mathbf T)$ is bipartite with
bipartition
\[
 L(\mathbf T)\ \sqcup\
 \bigl(\{2,\ldots,n\}\setminus L(\mathbf T)\bigr).
\]

The graph records all restrictions on synchronised orientations.
To make this precise, let $\mathcal U_{n,\ell}$ be the set of ordered
$\ell$-tuples of increasing rooted trees on $[n]$ having at most two
children at each vertex, with no left--right order.  Define $G(\mathbf
U)$ by the same sibling rule.  For a connected bipartite component
$C$, write $a_C,b_C$ for its two part sizes; for an isolated vertex,
these sizes are $0$ and $1$.

\begin{proposition}\label{tree:orientation-sum}
A tuple $\mathbf U\in\mathcal U_{n,\ell}$ admits a synchronised plane
orientation if and only if $G(\mathbf U)$ is bipartite.  Moreover,
\begin{equation}\label{tree:graph-enumerator}
 A_{n,\ell}(t)=
 \sum_{\substack{\mathbf U\in\mathcal U_{n,\ell}\\
                  G(\mathbf U)\text{ bipartite}}}
 \prod_{C\in\operatorname{Comp}(G(\mathbf U))}
 (t^{a_C}+t^{b_C}).
\end{equation}
\end{proposition}

\begin{proof}
We assign to each nonroot label its common left--right status in all
coordinates.  This assignment gives a valid plane orientation
precisely when siblings receive opposite statuses, that is, when it
is a proper two-colouring of $G(\mathbf U)$.  Each connected
bipartite component has exactly two such colourings, contributing
$t^{a_C}$ and $t^{b_C}$.  Choices on different components are
independent.  Summing over $\mathbf U$ proves the identity.
\end{proof}

For a connected component $C$ of $G(\mathbf T)$, define its
\emph{component flip} by reversing the left--right status of every
label in $C$, simultaneously in every coordinate.  More explicitly,
when a label in $C$ has a sibling in a given tree, that sibling also
belongs to $C$, and the two child subtrees are exchanged.  When it
has no sibling, its subtree is moved to the other side of its
parent.  This definition includes unary edges.  Parent relations
and labels do not change, so the sibling graph and its components
remain unchanged.  The flip is an involution on synchronised tuples.

We now specialise to $\ell=2$.

\begin{lemma}\label{fs:component-balance}
Every component of $G(\mathbf T)$ for $\mathbf T\in\mathcal T_{n,2}$
is a path or an even cycle.  An odd component of size $2a+1$ has
bipartition sizes $a,a+1$, and an even component of size $2b$ has
bipartition sizes $b,b$.
\end{lemma}

\begin{proof}
The graph is bipartite and is the union of two matchings, so its
maximum degree is at most two.  Its components are therefore paths
or even cycles.  Alternating the two parts along a path or cycle
gives the claimed sizes.
\end{proof}

For each $i\in\{2,\ldots,n\}$, define a map $\tau_i$ on the whole
set $\mathcal T_{n,2}$ as follows: if $i$ is the smallest label of an
odd component of $G(\mathbf T)$, flip that component; otherwise fix
$\mathbf T$.  In particular, even components are never flipped.
We call this the \emph{modified synchronised Foata--Strehl action}.

\begin{proof}[Proof of Theorem~\ref{main:gamma}\textup{(i)}: the action and orbit enumerator]
Component flips preserve the sibling graph, including its component
minima.  Thus applying $\tau_i$ twice restores the original tuple.
Two distinct nontrivial generators change the statuses of disjoint
sets of labels and hence commute.  Their actions remain independent
even when the corresponding child subtrees are nested, because the
underlying parent relations are fixed.

Let $s$ be the number of odd components and put
$k=\sum_{C\in\Comp(G(\mathbf T))}\lfloor |C|/2\rfloor$.
An odd component of size $2a+1$ contributes $t^a+t^{a+1}$ to the
orbit enumerator.  An even component of size $2b$ retains its
orientation and contributes $t^b$.  By
\Cref{fs:component-balance}, multiplication over components gives
$t^k(1+t)^s$.  Finally, summing the component sizes gives
$n-1=2k+s$.
\end{proof}

Call $\mathbf T\in\mathcal T_{n,2}$ \emph{reduced} if every odd
component $C$ satisfies
\[
 |C\cap L(\mathbf T)|=\frac{|C|-1}{2}.
\]
There is no additional restriction on even components.

\begin{proof}[Proof of Theorem~\ref{main:gamma}\textup{(i)}: the reduced representatives]
Each odd component has a unique orientation with its smaller part
on the left.  The modified action chooses these orientations
independently and fixes all even-component orientations.  Every
orbit therefore has exactly one reduced pair.  Its number of
common left-child labels is the exponent $k$ in
\eqref{fs:pair-orbit}.  Summing that identity over all orbits proves
\eqref{alg:gamma-expansion} for $\ell=2$, with the asserted counting
interpretation.  A tree without left children is the
unique right chain on $[n]$, so there is exactly one reduced pair
with no left children.
\end{proof}

The gamma-positivity of $A_{n,2}$ was proved algebraically by
Shankar~\cite[Proposition~27]{Shankar2026}.  Theorem~\ref{gamma:reduced-pairs}
gives a group action and a combinatorial interpretation of its gamma
coefficients.  When there is only one coordinate tree, the sibling
graph consists of isolated vertices and single edges.  Its odd
components are exactly the labels of unary children, so the same
construction reduces to the unary action above.

\begin{example}\label{ex:r2-orbit}
Take the pair $(T_1,T_2)$ in Figure~\ref{fig:r2-orbit}, with common
left-child set $\{2\}$.  Its sibling graph consists of the path
$3\! -\! 2\! -\! 4$ and the isolated vertex $5$.  Thus the two active
generators are $\tau_2$, which flips $C=\{2,3,4\}$, and $\tau_5$,
which flips $D=\{5\}$.  The original pair is reduced: on each odd
component, the left vertices occupy the smaller colour class.

Table~\ref{tab:r2-orbit} gives the four orbit elements through their
inverse images under $\Phi_5$ in each coordinate.  In every row the
common block-minimum set is $\{1\}\cup L$.  The orbit enumerator is
\[
 t+2t^2+t^3=t(1+t)^2,
\]
so this orbit contributes one to $\gamma_{5,2,1}$.
\end{example}
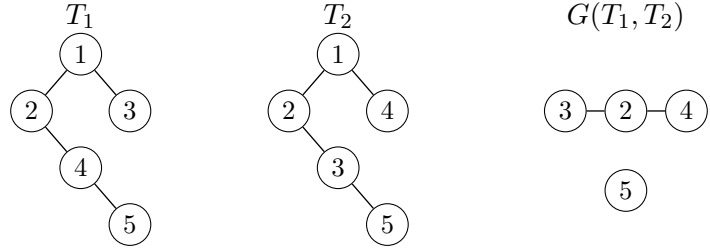
\begin{figure}[htbp]
\centering
\begin{tikzpicture}[x=1cm,y=1cm]
\begin{scope}[xshift=-4cm]
\node[treevertex] (a) at (0,0) {1};
\node[treevertex] (b) at (-.65,-.75) {2};
\node[treevertex] (c) at (.65,-.75) {3};
\node[treevertex] (d) at (0,-1.5) {4};
\node[treevertex] (e) at (.65,-2.25) {5};
\draw[treeedge] (a)--(b) (a)--(c) (b)--(d) (d)--(e);
\node at (0,.5) {$T_1$};
\end{scope}
\begin{scope}[xshift=-.6cm]
\node[treevertex] (a) at (0,0) {1};
\node[treevertex] (b) at (-.65,-.75) {2};
\node[treevertex] (c) at (.65,-.75) {4};
\node[treevertex] (d) at (0,-1.5) {3};
\node[treevertex] (e) at (.65,-2.25) {5};
\draw[treeedge] (a)--(b) (a)--(c) (b)--(d) (d)--(e);
\node at (0,.5) {$T_2$};
\end{scope}
\begin{scope}[xshift=3.2cm]
\node[treevertex] (a) at (-.8,-.75) {3};
\node[treevertex] (b) at (0,-.75) {2};
\node[treevertex] (c) at (.8,-.75) {4};
\node[treevertex] (d) at (0,-1.8) {5};
\draw[treeedge] (a)--(b)--(c);
\node at (0,.5) {$G(T_1,T_2)$};
\end{scope}
\end{tikzpicture}
\caption{A reduced synchronised pair for $\ell=2$ and its sibling graph.
In each tree, a child drawn to the left or right occupies that child slot.}
\label{fig:r2-orbit}
\end{figure}
\begin{table}[htbp]
\centering
\begin{tabular}{@{}llll@{}}
\toprule
Applied flips & Common left-child set $L$ & Function pair & $|L|$\\
\midrule
None & $\{2\}$ & $(12122,12212)$ & $1$\\
$\tau_5$ & $\{2,5\}$ & $(12125,12214)$ & $2$\\
$\tau_2$ & $\{3,4\}$ & $(11233,11323)$ & $2$\\
$\tau_2\tau_5$ & $\{3,4,5\}$ & $(11235,11324)$ & $3$\\
\bottomrule
\end{tabular}
\caption{The four elements of the orbit in Example~\ref{ex:r2-orbit}.}
\label{tab:r2-orbit}
\end{table}

\subsection{A recurrence for the gamma coefficients}

Expanding the differential recurrence of
Shankar~\cite[Proposition~26]{Shankar2026} in the gamma basis gives the
two-term recurrence in Theorem~\ref{main:gamma}\textup{(ii)} for the reduced-pair numbers.  We give
the calculation in its univariate form.

\begin{proof}[Proof of Theorem~\ref{main:gamma}\textup{(ii)}]
Let $\theta=t\,d/dt$.  The defining recurrence for the coefficients
of $A_{n,2}$ gives
\[
 A_{n,2}=
 \bigl((\theta+1)^2+t(n-1-\theta)^2\bigr)A_{n-1,2}.
\]
For $j,q\ge0$ with $2j+q=n-2$, differentiation gives the basis identity
\begin{align}
 &\bigl((\theta+1)^2+t(n-1-\theta)^2\bigr)
       t^j(1+t)^q \notag\\
 &\qquad=(j+1)^2t^j(1+t)^{q+1}
       +q(q+4j+5)t^{j+1}(1+t)^{q-1},
 \label{gamma:basis-identity}
\end{align}
where the second summand is omitted when $q=0$.  For completeness,
after factoring out $t^j$, the operator on the left is
\[
 (\theta+j+1)^2+t(q+j+1-\theta)^2.
\]
Substituting
\[
 \theta(1+t)^q=qt(1+t)^{q-1},\qquad
 \theta^2(1+t)^q=qt(1+t)^{q-1}
                  +q(q-1)t^2(1+t)^{q-2}
\]
and collecting terms proves \eqref{gamma:basis-identity}; the cases
$q=0,1$ follow either directly or by omitting the vanishing derivative
terms.  We now apply the identity to the gamma expansion of $A_{n-1,2}$.
The coefficient of $t^k(1+t)^{n-1-2k}$ receives the first contribution
from $j=k$ and the second from $j=k-1$.  In the latter case,
$q=n-2k$ and $q+4j+5=n+2k+1$, giving
\eqref{gamma:recurrence-equation}.
\end{proof}

\section{PECKness and combinatorial \texorpdfstring{$\mathfrak{sl}_2$}{sl2} actions}\label{sec:PECK}
We prove Theorem~\ref{main:peck} and discuss Conjecture~\ref{PECK:conjecture}.

We return to the posets defined in Section~\ref{sec:posets}.
The kernel-preserving tree bijection makes it possible to construct linear
operators supported on their cover relations when $\ell=1$.
We first recall the representation-theoretic mechanism that turns these
operators into a proof of the strong Sperner property.

\subsection{The combinatorial \texorpdfstring{$\mathfrak{sl}_2$}{sl2} mechanism}

Let $P_k$ denote the set of elements of rank $k$ in a finite graded poset $P$. A finite graded poset $P$ of rank $N$ is \emph{rank-symmetric} if
$|P_k|=|P_{N-k}|$ for every $k$, and \emph{rank-unimodal} if its rank sizes
are unimodal.  It is \emph{strongly Sperner} if, for every
$1\leq s\leq N+1$, the largest union of $s$ antichains has size equal to
the sum of its $s$ largest rank sizes.  A poset having all three
properties is called a \emph{PECK poset}.

Let $V_k=\mathbb C P_k$ be the vector space with basis the elements of
$P_k$, and put $V=\bigoplus_{k=0}^N V_k$.  A linear map $E:V\to V$ is
\emph{order-raising} if
\[
 E x\in\operatorname{span}\{y\in P:y\gtrdot x\}
 \qquad(x\in P).
\]
The linear-algebraic approach to Sperner theory replaces matchings between
ranks by order-raising maps of full rank.  Stanley's textbook
\cite[Chapter~4]{StanleyAlgComb2013} explains this approach through the up
and down operators on Boolean lattices.  For strong Sperner properties,
we use the following criterion.

\begin{theorem}[{\cite[Lemma~1.1]{Stanley1980}}]
\label{PECK:raising-criterion}
A finite graded poset $P$ of rank $N$ is PECK if and only if there is an
order-raising map $E$ such that
\begin{equation}\label{PECK:lefschetz}
 E^{N-2k}:V_k\longrightarrow V_{N-k}
 \quad\text{is an isomorphism for }0\leq k\leq\lfloor N/2\rfloor.
\end{equation}
\end{theorem}

The Lie algebra $\mathfrak{sl}_2(\mathbb C)$ has generators $e,f,h$ with
$[h,e]=2e$, $[h,f]=-2f$, and $[e,f]=h$, where
$[X,Y]=XY-YX$.  Its finite-dimensional representations are direct sums
of irreducible modules; see \cite[Sections~6--7]{Humphreys1972}.
For each integer $d\geq0$, the irreducible module of highest weight $d$
has a basis $v_0,\ldots,v_d$ in which
\begin{equation}\label{PECK:irreducible-model}
 h v_i=(2i-d)v_i,\qquad
 e v_i=(d-i)v_{i+1},\qquad
 f v_i=i v_{i-1},
\end{equation}
with the terms outside the index range interpreted as zero.
Thus the weights form a string symmetric about zero, and powers of $e$
identify opposite weight spaces.  Assigning weight $2k-N$ to rank $k$
turns these identifications into the isomorphisms in
\eqref{PECK:lefschetz}.  See, for example, Proctor \cite{Proctor1982} and
Gaetz--Gao \cite{GaetzGao2020}. The following lemma gives another sufficient condition for PECKness. We include the proof here to keep the article self-contained.

\begin{lemma}\label{PECK:sl2-criterion}
Suppose that $E$ is order-raising, that $F(V_k)\subseteq V_{k-1}$, and
that $H|_{V_k}=(2k-N)I$.  If
\begin{equation}\label{PECK:commutators}
 [H,E]=2E,\qquad [H,F]=-2F,\qquad [E,F]=H,
\end{equation}
then $P$ is PECK.
\end{lemma}

\begin{proof}
The relations give an $\mathfrak{sl}_2$-representation on $V$.
In each irreducible summand, \eqref{PECK:irreducible-model} shows that
$E^m$ identifies the weight $-m$ and weight $m$ spaces for every integer
$m\geq0$; both spaces are zero when these weights do not occur.
Taking direct sums and setting $m=N-2k$ gives
\eqref{PECK:lefschetz}, so Theorem~\ref{PECK:raising-criterion} applies.
\end{proof}

\subsection{Linear operators for \texorpdfstring{$\ell=1$}{ell=1}}

We use the bijection $\Phi_n$ of Theorem~\ref{tree:kernel-bijection} to
identify the basis of $\mathbb C\mathcal Q_{n,1}$ with the increasing
binary trees in $\mathcal T_n$.  The rank of a tree is its number of
left edges, and its maximal right chains give the blocks of the
corresponding kernel.

For a tree $T$, let $R_u(T)$, respectively $L_u(T)$, be the set of labels
of right, respectively left, children whose parents have exactly one
child.  If $i\in R_u(T)\cup L_u(T)$, write $\sigma_iT$ for the tree
obtained by moving the subtree rooted at $i$ to the opposite child slot
of its parent.  Define
\begin{equation}\label{PECK:tree-operators}
 ET=\sum_{i\in R_u(T)}\sigma_iT,
 \qquad
 FT=\sum_{i\in L_u(T)}\sigma_iT,
 \qquad
 HT=(2\rho(T)-n+1)T.
\end{equation}

\begin{proof}[Proof of Theorem~\ref{main:peck}: the operators and PECKness]
Moving a unary right edge to the left splits one maximal right chain
into two and leaves the other right chains unchanged.  By the
kernel-preserving property of $\Phi_n$, this splits one kernel block.
Theorem~\ref{main:models}\textup{(ii)} shows that every summand of $ET$ is an
upper cover of $T$.  Reversing the move merges two right chains, so every
summand of $FT$ is a lower cover.

To verify the commutators, we fix the parent relations of a tree and the
left--right order at every vertex having two children, and forget only
the orientations at vertices having one child.  We call the resulting
object a skeleton.  Suppose that it has $b$ vertices with two children
and $u$ unary edges.  Counting edges gives
\[
 n-1=2b+u.
\]
We index the unary edges by $[u]$.  The orientations of this skeleton have
basis $e_S$, $S\subseteq[u]$, where $S$ records the edges oriented to the
left.  These basis vectors have rank $b+|S|$, and the operators restrict
to
\begin{align}
 Ee_S&=\sum_{i\notin S}e_{S\cup\{i\}},&
 Fe_S&=\sum_{i\in S}e_{S\setminus\{i\}},&
 He_S&=(2|S|-u)e_S.\label{PECK:boolean-operators}
\end{align}
The first two relations in \eqref{PECK:commutators} follow because $E$
raises $|S|$ by one and $F$ lowers it by one.  In $(EF-FE)e_S$, the terms
changing two distinct edges cancel in pairs.  The coefficient of
$e_S$ is $|S|-(u-|S|)=2|S|-u$, proving $[E,F]=H$.

Each skeleton spans an invariant subspace, and the skeleton classes
partition the tree basis.  The relations therefore hold on the whole
space.  We can now apply Lemma~\ref{PECK:sl2-criterion} to obtain PECKness.
\end{proof}

The same construction also gives a partition into symmetric Boolean
subposets.  A \emph{symmetric Boolean decomposition} of a graded poset
of rank $N$ is a partition into Boolean subposets whose ranks agree with
the ambient ranks up to a shift and whose minimum and maximum ranks
sum to $N$.  A \emph{symmetric chain decomposition} is a partition into
saturated chains with the same condition on their endpoint ranks.

\begin{proof}[Proof of Theorem~\ref{main:peck}: the symmetric decompositions]
We fix a skeleton with $b$ binary vertices and $u$ unary edges.
The orientations of a fixed skeleton form a Boolean subposet, with
$S\subseteq S'$ corresponding to changing unary right edges to left
edges.  Conversely, comparability of two orientations implies inclusion
of their left-child label sets, hence inclusion of the corresponding
subsets of unary edges.  Thus the induced order is Boolean.
Its ranks range from $b$ to $b+u$, and $2b+u=n-1$ makes it symmetric.
These subposets partition $\mathcal Q_{n,1}$.

For completeness, a symmetric chain decomposition of a Boolean lattice
can be constructed inductively.  We start with the one-element lattice
$\mathcal B_0$.  Given a symmetric chain $x_0<\cdots<x_m$ in
$\mathcal B_u$, we partition its product with $\{0,1\}$ into
\[
 (x_0,0)<\cdots<(x_m,0)<(x_m,1),
 \qquad
 (x_0,1)<\cdots<(x_{m-1},1),
\]
omitting the second chain if $m=0$.  Both chains are saturated and
symmetric in $\mathcal B_{u+1}$.  Applying this to every chain and
then to each skeleton class proves the second assertion.
\end{proof}

\begin{samepage}
\begin{example}\label{ex:scd}
The poset $\mathcal Q_{4,1}$ has rank sizes $(1,11,11,1)$.  A symmetric
chain decomposition consists of the chain
\[
 1111\lessdot1114\lessdot1134\lessdot1234
\]
and the following ten chains:
\[
\begin{array}{@{}c@{\qquad}c@{}}
1212\lessdot1213 & 1211\lessdot1214\\[2pt]
1221\lessdot1231 & 1131\lessdot1232\\[2pt]
1222\lessdot1224 & 1133\lessdot1233\\[2pt]
1113\lessdot1223 & 1121\lessdot1123\\[2pt]
1122\lessdot1124 & 1112\lessdot1132.
\end{array}
\]
The first chain has endpoint ranks $0$ and $3$; every other chain has
endpoint ranks $1$ and $2$.  All $24$ subexceedant functions occur
exactly once.  Each step splits one kernel block.  For example,
$1113\lessdot1223$ splits $\{1,2,3\}\mid\{4\}$ into
$\{1\}\mid\{2,3\}\mid\{4\}$.
\end{example}
\end{samepage}

\subsection{The PECK conjecture for higher parameters}

The rank polynomial $A_{n,\ell}(t)$ is palindromic by
Proposition~\ref{prelim:palindromicity} and real-rooted by
Theorem~\ref{main:eulerian}\textup{(i)}.  Consequently,
$\mathcal Q_{n,\ell}$ is rank-symmetric and rank-unimodal for every $n,\ell$.
The remaining PECK condition concerns its order.

Conjecture~\ref{PECK:conjecture} asserts that $\mathcal Q_{n,\ell}$ is PECK for every $\ell>1$.

Equivalently, it remains to prove that these posets are strongly
Sperner.  The component flips used in the $\ell=2$ group action do not
directly give the required order-raising operators.  Flipping an odd
path with at least three vertices replaces one nonempty set of
left-child labels by another disjoint nonempty set.  The two
block-minimum sets are then incomparable by inclusion, so the two
tuples are incomparable in $\mathcal Q_{n,2}$.

There is computational evidence for the stronger assertion that
$\mathcal Q_{n,\ell}$ admits a symmetric chain decomposition.  If $P_k$ and
$P_{k+1}$ are consecutive ranks and $N(X)$ is the set of upper neighbours
of $X\subseteq P_k$, the \emph{normalised matching property} is
\begin{equation}\label{PECK:normalized-matching}
 \frac{|N(X)|}{|P_{k+1}|}\geq\frac{|X|}{|P_k|}
 \qquad(X\subseteq P_k).
\end{equation}
A rank-symmetric, rank-unimodal poset satisfying this property at every
pair of consecutive ranks admits a symmetric chain decomposition
\cite{Griggs1977}.  We have verified
\eqref{PECK:normalized-matching} throughout the following ranges:
\[
 \begin{array}{c|ccc}
 \ell&2&3&4\\ \hline
 n&1\leq n\leq8&1\leq n\leq7&1\leq n\leq6.
 \end{array}
\]
For $\ell=2$, explicit symmetric chain decompositions have also been
constructed and checked through $n=7$.

\section{Extensions to Stirling permutations}\label{sec:stirling}
We prove Theorems~\ref{main:stirling} and~\ref{stirling:plateau-peck}.

The same normalisation applies to descent numbers of Stirling
permutations.  We distinguish the multiplicity $m$ of each letter from
the synchronisation parameter $\ell$.  Ordinary Stirling permutations
correspond to $m=2$, whereas $m=1$ gives ordinary permutations.

\subsection{Descent histories and stability}

An \emph{$m$-Stirling permutation} of order $n$ is a permutation of
$\{1^m,\ldots,n^m\}$ in which no smaller letter occurs between two
consecutive occurrences of a letter.  We count only internal descents:
\[
 d(\sigma)=|\{i:1\leq i<mn,\ \sigma_i>\sigma_{i+1}\}|.
\]
Thus the customary mandatory descent to a terminal zero is omitted.
If $\sigma^{[i]}$ is obtained by deleting all letters greater than $i$,
put $e_i=d(\sigma^{[i]})-d(\sigma^{[i-1]})$ for $2\leq i\leq n$.
The largest letter is inserted as one consecutive block, so $e_i$ is
either zero or one.  We call $(e_2,\ldots,e_n)$ its \emph{descent history}.

For a history $e$, set $k_i=e_2+\cdots+e_{i-1}$ and
\begin{equation}\label{stirling:history-weight}
 c_m(e)=\prod_{i=2}^n
 \begin{cases}
 k_i+1,&e_i=0,\\
 m(i-1)-k_i,&e_i=1.
 \end{cases}
\end{equation}
Indeed, a word on $[i-1]$ with $k_i$ descents has $k_i+1$
descent-preserving insertion gaps: its descent gaps and its final gap.
The remaining $m(i-1)-k_i$ gaps create a descent.  Consequently,
$c_m(e)$ counts the words with history $e$.

An ordered $\ell$-tuple is synchronised when its coordinates have the
same descent history.  Define
\begin{equation}\label{stirling:polynomials}
 F^{(m)}_{n,\ell}(\mathbf z)
 =\sum_{e\in\{0,1\}^{n-1}}c_m(e)^\ell\prod_{i:e_i=1}z_i,
 \qquad
 S^{(m)}_{n,\ell}(t)=F^{(m)}_{n,\ell}(t,\ldots,t).
\end{equation}
Writing $S^{(m)}_\ell(n,k)=[t^k]S^{(m)}_{n,\ell}(t)$ gives
\begin{equation}\label{stirling:recurrence}
 S^{(m)}_\ell(n,k)
 =(k+1)^\ell S^{(m)}_\ell(n-1,k)
 +(m(n-1)-k+1)^\ell S^{(m)}_\ell(n-1,k-1),
\end{equation}
with $S^{(m)}_\ell(1,0)=1$ and zero entries outside $0\leq k\leq n-1$.
For $m=1$, these are the parametrised Eulerian polynomials.

\begin{proof}[Proof of Theorem~\ref{main:stirling}\textup{(i)}]
Stability follows from Theorem~\ref{stab:general} and the
$m$-Stirling row of Table~\ref{stab:applications}; its balancing weights
are $w_n(k)=\binom{m(n-1)}k$.

For strict interlacing, the initial nonconstant polynomial is
$S^{(m)}_{2,\ell}=1+m^\ell t$.  Suppose that
$f=S^{(m)}_{n,\ell}$ has simple negative zeros
$\alpha_1<\cdots<\alpha_d$, where $d=n-1$, and put
\[
 U=(\theta+1)^\ell f,\qquad
 V=(mn-\theta)^\ell f,\qquad \theta=t\frac{d}{dt}.
\]
Then $S^{(m)}_{n+1,\ell}=U+tV$.  Since $mn>d$,
Lemma~\ref{int:root-movement}, applied successively, gives simple
negative zeros $u_i$ and $v_i$ satisfying $v_i<\alpha_i<u_i$.
Specializing the old variables of $F^{(m)}_{n+1,\ell}$ to $t$ gives
the stable polynomial $U(t)+zV(t)$.  Hence its coefficient polynomials
weakly interlace, in the order
\[
 v_1\leq u_1\leq v_2\leq u_2\leq\cdots\leq v_d\leq u_d.
\]
Both leading coefficients are positive.  It follows that
$\operatorname{sgn}U(\alpha_i)=(-1)^{d-i+1}$ and
$\operatorname{sgn}V(\alpha_i)=(-1)^{d-i}$.
Since $\alpha_i<0$, the two summands of
$U(\alpha_i)+\alpha_iV(\alpha_i)$ have the same nonzero sign.
Together with the positive constant and leading coefficients, this
places one zero in each of
\[
 (-\infty,\alpha_1),\quad
 (\alpha_i,\alpha_{i+1})\ (1\leq i<d),\quad (\alpha_d,0).
\]
Degree counting proves simplicity and strict interlacing.
The Strong Rayleigh, log-concavity, and unimodality assertions follow
from stability and its univariate specialisation, as in
Section~\ref{alg:section}.

The proof of Theorem~\ref{main:eulerian}\textup{(iii)} also shows that
$S^{(m)}_{n+1,\ell}(t)+zS^{(m)}_{n,\ell}(t)$ is real stable.
The only change is that the linear part of the quotient of consecutive
rows has slope equal to the positive ratio of their leading
coefficients.  Every real combination
$S^{(m)}_{n+1,\ell}+sS^{(m)}_{n,\ell}$ therefore has simple real zeros;
all are negative for $s>-1$, one is zero for $s=-1$, and exactly one
is positive for $s<-1$.
\end{proof}

At $\ell=1$, real-rootedness belongs to the classical theory; see
Brenti~\cite{Brenti1989} and B\'ona~\cite{Bona2008}.
Haglund--Visontai~\cite{haglund-visontai-stable-multivariate} proved stability for
multivariate refinements marking ascent, descent, and plateau values.
The polynomial \eqref{stirling:polynomials} records insertion histories,
which differ from the final descent-top sets.

\begin{remark}\label{stirling:variable-multiplicities}
The proof also permits positive multiplicities $m_1,\ldots,m_n$.
We put $C_1=0$ and $C_i=m_1+\cdots+m_{i-1}$ for $i\ge2$ and replace $m(i-1)$ by $C_i$ in
\eqref{stirling:history-weight}.  The balancing weights are
$\binom{C_n}k$, and the normalised recurrence becomes
\[
 P_{n,\ell}^{C}
 =(1+z_n)
 \left[
 (\Theta+1)\prod_{j=C_{n-1}+1}^{C_n}\frac{j-\Theta}{j}
 \right]^\ell P_{n-1,\ell}^{C}.
\]
Every displayed factor preserves stability by
Lemma~\ref{stab:euler-operators}, and the binomial multiplier restores
the unnormalized polynomial.  Since $C_{n+1}\geq n$, the interlacing
proof is unchanged.  Thus all the conclusions of
Theorem~\ref{main:stirling}\textup{(i)} hold for a fixed infinite sequence of
positive multiplicities.
\end{remark}

\subsection{Trees and the kernel order}

For constant multiplicity $m$, replace subexceedant functions by maps
$f:[n]\to[m(n-1)+1]$ satisfying
$f(i)\leq m(i-1)+1$.  Their kernels and block-leader sets are defined
as before.  Let $\mathcal R^{(m)}_{n,\ell}$ be their synchronised
$\ell$-tuples, ordered by proper kernel refinement in every coordinate,
as in Definition~\ref{poset:definition}.

\begin{proof}[Proof of Theorem~\ref{main:stirling}\textup{(ii)}]
We number the child slots $0,\ldots,m$, with $m$ the rightmost slot.
We fix a partition with block minima $1=b_1<\cdots<b_s$.
After choosing the distinct images of its first $j-1$ blocks, the
number of available images for the next block is
\[
 m(b_j-1)+1-(j-1)=m(b_j-1)-j+2.
\]
We form the prescribed increasing rightmost chains and attach their roots
$b_j$, for $j\geq2$, in increasing order to unused nonrightmost slots
at vertices smaller than $b_j$.  There are $m(b_j-1)$ such slots before
attachments and $j-2$ have already been used.  We match available images,
in increasing order, with available slots, ordered first by parent
label and then by slot number.  The counts agree.  All edges increase
labels, and the construction is invertible by recovering the
rightmost chains and reversing these choices.

The Gessel correspondence reads a vertex $a$ with subtrees
$T_0,\ldots,T_m$ as
\[
 W(T)=W(T_0)\,a\,W(T_1)\,a\,\cdots\,a\,W(T_m),
 \qquad W(\varnothing)=\varnothing;
\]
see \cite[Theorems~1--2]{JansonKubaPanholzer2011}.
The labels in a nonempty $T_j$, $j<m$, are larger than $a$, so the
last letter of $W(T_j)$ followed by $a$ contributes one descent.
The intervening transitions contribute no additional descents.
Induction therefore identifies the number of internal descents with
the number of nonrightmost edges.  Deleting the largest vertex simply
removes its incident edge, so the labels of these edges record the
descent history.  They are precisely the nonroot right-chain minima.

The grading argument of Theorem~\ref{main:models}\textup{(ii)} applies:
intermediate synchronised partitions can be realised by giving each
block its minimum as image, which satisfies the enlarged bounds.
The rank is the common number of blocks minus one.  The history
enumeration proves the asserted rank polynomial.
\end{proof}

\subsection{Fixed plateaux, partial gamma-positivity, and PECKness}

We now take $m=2$ and $\ell=1$.  For an ordinary Stirling permutation,
let $\operatorname{plat}(\sigma)$ count the adjacent equal pairs.
In its ternary tree this equals the number of empty middle slots.
For $1\leq p\leq n$, let $\mathcal R_{n,p}$ be the induced subposet
of $\mathcal R^{(2)}_{n,1}$ whose corresponding words have $p$ plateaux.
Equivalently, their trees have exactly $q=n-p$ middle edges.

The outer-slot involutions below realise the same orbit decomposition
as the Gessel-tree Foata--Strehl action of
Chen--Fu--Yan~\cite[Section~3 and Theorem~3.1]{ChenFuYan2023}.
Its partial gamma-positive enumerator is known; see also
Ma--Ma--Yeh~\cite[Section~3]{MaMaYeh2019}.
Its compatibility with the kernel order yields
Theorem~\ref{stirling:plateau-peck}, which we now prove.

\begin{proof}[Proof of Theorem~\ref{stirling:plateau-peck}]
A partition with $s$ blocks is the right-chain partition of a tree
with $q$ middle edges if and only if $s\geq q+1$.  Necessity holds
because every middle child starts a nonroot right chain.  Conversely,
we choose $q$ nonroot block minima to be middle children and designate
all the remaining nonroot block minima as left children.  We attach the
chain roots in increasing order.  At a root $v$, its designated slot
is available at $v-1$ earlier vertices, and at most $v-2$ previously
attached roots have used that slot type.  An available parent
therefore exists.  This realizes the partition with exactly $q$
middle edges.

It follows that the possible block counts are $q+1,\ldots,n$.
Any intermediate partition between two comparable elements is
realizable with $q$ middle edges.  Hence covers are precisely
comparable pairs whose block counts differ by one.  A partition with
more than $q+1$ blocks can be coarsened by one merge and realised
again; a partition with fewer than $n$ blocks can be refined by one
split.  Thus all minimal elements have $q+1$ blocks, all maximal
elements have $n$ blocks, and the stated grading follows.

We fix the parent relations of a ternary tree and every middle edge.
We also fix the left--right order at vertices having both outer slots
occupied.  At a vertex with exactly one occupied outer slot, we forget
whether its outer child is left or right; a middle child may still be
present.  The resulting object is called an \emph{outer-slot skeleton}.
If it has $b$ doubly occupied outer pairs and $u$ singly occupied
outer pairs, then
\[
 n-1=q+2b+u.
\]
Its orientations are indexed by $S\subseteq[u]$, recording the outer
edges directed left.  They have $q+b+|S|$ internal descents and hence
rank $b+|S|$ in $\mathcal R_{n,p}$.

Moving a singly occupied outer edge from right to left deletes one
right-chain edge and splits one kernel block.  It is therefore an
upper cover in the induced poset.  Conversely, comparability between
two orientations forces inclusion of their block-minimum sets, hence
inclusion of the corresponding subsets $S$.  Each skeleton class is
an induced Boolean subposet.  Its endpoint ranks sum to
$2b+u=p-1$, so these classes form a symmetric Boolean decomposition.
Applying the Boolean symmetric chain construction from
Theorem~\ref{main:peck} gives a symmetric chain
decomposition of $\mathcal R_{n,p}$.

For the operator statement, we use the Boolean
operators in \eqref{PECK:boolean-operators} on each skeleton.  Their diagonal operator
satisfies
\[
 H e_S=(2|S|-u)e_S
       =(2\rho_p(T)-(p-1))e_S.
\]
The raising and lowering operators are supported on the covers just
described.  Their direct sums satisfy
\eqref{PECK:commutators}, so we can apply Lemma~\ref{PECK:sl2-criterion} to obtain
PECKness.  Finally, a skeleton contributes
$t^b(1+t)^u$ to the shifted descent enumerator, giving
\eqref{stirling:partial-gamma}.
\end{proof}

\begin{example}
Take a ternary tree whose root $1$ has middle child $2$, and whose
vertex $2$ has right child $3$.  Its word is $122331$.  Moving the
child $3$ to the left produces $133221$.  The two words have two
plateaux and respectively one and two internal descents, so they
form a two-element chain in $\mathcal R_{3,2}$ with shifted ranks
zero and one.  The three plateau slices at $n=3$ have descent
enumerators
\[
 \begin{array}{c|ccc}
 p&3&2&1\\\hline
 \displaystyle\sum_{\operatorname{plat}(\sigma)=p}t^{d(\sigma)}
 &1+4t+t^2&t(4+4t)&t^2.
 \end{array}
\]
Their sum is $1+8t+6t^2$.  The separate symmetry centers explain why
the full descent polynomial is not palindromic.
\end{example}

\begin{remark}
The fixed-plateau theorem concerns one Stirling permutation.
For synchronised tuples, a common descent history specifies whether
each inserted label uses a left or middle slot, but does not specify
which of the two.  Thus simultaneous outer-slot flips need not
preserve this family.  For example, the pair $(2211,1221)$ has common
history $(1)$; flipping its available outer edge gives
$(1122,1221)$, with different histories.  The stability theorem for
arbitrary $\ell$ therefore does not imply a corresponding
fixed-plateau PECK or gamma theorem for synchronised tuples.
\end{remark}
\section*{Declaration of AI usage}
AI tools have been used for coding assistance for data collection, symbolic calculations, and language editing. AI tools were not used to find proofs of the results in this paper.

\end{document}